\documentclass[10pt,centertags,reqno]{amsart}

\usepackage{latexsym}
\usepackage[english]{babel}
\usepackage[T1]{fontenc}
\usepackage{amssymb}
\usepackage{fancyhdr}
\usepackage{url}
\usepackage{hyperref}
\usepackage{verbatim}
\usepackage{leftidx}

\numberwithin{equation}{section} \makeatletter
\renewcommand{\subsection}{\@startsection
{subsection}{2}{0mm}{\baselineskip}{-0.25cm}
{\normalfont\normalsize\bf}} \makeatother

\newtheorem{theorem}{Theorem}[section]
\newtheorem{lemma}[theorem]{Lemma}
\newtheorem{corollary}[theorem]{Corollary}
\newtheorem{definition}[theorem]{Definition}
\newtheorem{remark}[theorem]{Remark}
\newtheorem{proposition}[theorem]{Proposition}

\def \A {\mathcal A}
\def \F {\mathcal F}

\def \H {\mathcal H}
\def \C {\mathcal C}

\def \N {\mathcal N}
\def \L {\mathcal L}
\def \P {\mathbf P}

\def \R {\mathbb R}
\def \E {\mathcal E}
\def \I {{\mathbf 1}}
\def \bF {\mathbb F}

\def \bH {\mathbb H}
\def \bE {\mathbb E}

\newcommand{\ud}{\mathrm d}
\newcommand{\ds}{\displaystyle}
\newcommand{\esp}[2][\mathbb E] {#1\left[#2\right]}
\newcommand{\espp}[2][\mathbb E^{\mathbf P^*}] {#1\left[#2\right]}

\newcommand{\condesph}[2][\H_t]       {\mathbb E\left.\left[#2\right|#1\right]}
\newcommand{\condespHH}[2][\H_t]       {\mathbb E^{\mathbf P^*}\left.\left[#2\right|#1\right]}

\newcommand{\condesphoo}[2][\H_{0}]       {\mathbb E^{\mathbf P^*}\left.\left[#2\right|#1\right]}

\newcommand{\condesphTT}[2][\H_{T}] {\mathbb E^{\mathbf P^*} \left.\left[#2\right|#1\right]}
\newcommand{\doleans}[1] {\mathcal E\left(#1\right)}

\author[C.~Ceci]{Claudia  Ceci}
\author[K.~Colaneri]{Katia Colaneri}
\author[A.~Cretarola]{Alessandra Cretarola}

\begin{document}
\address{Claudia  Ceci, Department of Economics,
University ``G. D'Annunzio'' of Chieti-Pescara, Viale Pindaro, 42,
I-65127 Pescara, Italy.}\email{c.ceci@unich.it}

\address{Katia Colaneri, Department of Economics,
University ``G. D'Annunzio'' of Chieti-Pescara, Viale Pindaro, 42,
I-65127 Pescara, Italy.}\email{katia.colaneri@unich.it}

\address{Alessandra Cretarola, Department of Mathematics and Computer Science,
 University of Perugia, Via Vanvitelli, 1, I-06123 Perugia, Italy.}\email{alessandra.cretarola@unipg.it}

\title[LRM under restricted information on asset prices]{Local risk-minimization under restricted information on asset prices}

\date{}

\begin{abstract}
\begin{center}
In this paper we investigate the local risk-minimization approach for a semimartingale financial market where there are restrictions on the available information to
agents who can observe at least the asset prices.
We characterize the optimal strategy in terms of suitable decompositions of a given contingent claim, with respect to a filtration representing the information level, even in presence of jumps. Finally, we discuss some practical examples in a Markovian framework and
show that the computation of the optimal strategy leads to filtering problems under the real-world probability measure and under the minimal
martingale measure.
\end{center}
\end{abstract}

\maketitle

{\bf Keywords}: Local risk-minimization; partial information;  Markovian processes; filtering.

{\bf AMS MSC 2010}: Primary: 60J25; 60G35; 91B28; Secondary: 60J75; 60J60.

\section{Introduction}

In this paper we study the {\em local risk-minimization} approach (see e.g.~\cite{fs1991}, ~\cite{s93}
and~\cite{s01} for a deeper discussion on this issue)
 for a semimartingale market model where there are restrictions on the available information
to traders and discuss some Markovian models where we compute explicitly the optimal strategy even by means of filtering problems. \\
More precisely, we assume that in our model the agents have a limitative knowledge on the market, so that their choices cannot be based on the full
information flow described by the filtration $\bF:=\{\F_t, \ t \in [0,T]\}$, with $T$ denoting a fixed finite time horizon. The available information
level is basically given by a smaller filtration $\bH:=\{\H_t, \ t \in [0,T]\}$. However, since, in general, stock prices are publicly available, we assume that the agents can reasonably observe at least the asset prices.

In this market we consider a European-type contingent claim whose final payoff is given by an $\H_T$-measurable square-integrable random variable $\xi$
on the given probability space $(\Omega,\F,\P)$. The  goal is to study the hedging problem of the payoff $\xi$ via the local risk-minimization approach in the underlying incomplete market, which is driven by an $(\bF,\P)$-semimartingale $S$ representing the stock
price process and where there are restrictions on the available information to traders.

The quadratic hedging method of local risk-minimization extends the theory of risk-minimization introduced in \cite{fs86} and formulated when the price process is a local martingale under the real-word probability measure $\P$, to the semimartingale case. The local martingale case was largely developed both under complete and partial information. One of the pioneer papers in the restricted information setting is represented by \cite{s94}, where the optimal strategy is constructed via predictable dual projections. More recently, in \cite{ccr1}, the authors characterized the risk-minimizing hedging strategy via an orthogonal decomposition of the contingent claim, called the {\em Galtchouk-Kunita-Watanabe decomposition under restricted information}.

The local risk-minimization method under partial information has been investigated  for the first time in~\cite{ccr2}, where the authors, thanks to existence and uniqueness results for backward stochastic differential equations under partial information, characterize the optimal hedging strategy for an $\F_T$-measurable contingent claim $\xi$, via a suitable version of the F\"ollmer-Schweizer
decomposition working in the case of restricted information, by means of the new concept of weak orthogonality introduced in \cite{ccr1}. More precisely, they prove that the $\bH$-predictable integrand with respect to the stock price process in the F\"ollmer-Schweizer decomposition gives the $\bH$-locally risk minimizing strategy; nevertheless, they do not furnish any operational method to represent explicitly the optimal strategy.  Our contribution, in this context, is to provide a full description of the optimal strategy for an $\H_T$-measurable contingent claim, under the additional hypothesis that the information available to investors is, at least, given by the stock prices. This scenario is characterized by the following condition on filtrations:
\[
\F^S_t \subseteq \H_t \subseteq \F_t, \quad t \in [0,T],
\]
where $\F^S_t$ is the $\sigma$-field generated by the stock price process $S$ up to time $t$.\\
In this paper, the key point is that the risky asset price process $S$ satisfying the {\em structure condition} with respect to $\bF$, see \eqref{eq:SC},
turns out to be an $(\bH,\P)$-semimartingale in virtue of the condition above. Indeed, since the payoff of a given contingent claim is always supposed to be an $\H_T$-measurable random variable, this allows one to
reduce the hedging problem under partial information to an equivalent problem in the case of full information, as all involved processes turn out to be
$\bH$-adapted. We will see that $S$ also satisfies the structure condition with respect to $\bH$, see Proposition \ref{prop:N}, and then the optimal strategy can be characterized  by extending the results of \cite{cvv2010} to the partial information framework, see Proposition \ref{casoJ}. The Galtchouk-Kunita-Watanabe under restricted information, with respect to the {\em minimal martingale measure}, $\P^*$, see Definition \ref{def:MMM}, represents an essential tool to get the achievement.

We also pay attention to the relationship between the optimal strategy under complete information and that under restricted information. In Proposition \ref{prop:relazione_strategie} the result is stated under the assumption that the stock price process has continuous trajectories, and then generalized to the discontinuous case in Proposition \ref{casoJ}.

Finally, we consider some Markovian models affected by an unobservable stochastic factor.
We discuss three meaningful examples where we characterize the structure conditions of the underlying price process with respect to both $\bF$ and $\bH$ and compute the optimal strategy when the information flow coincides with the natural filtration for the stock price process.
In the first example, $S$ is a geometric diffusion process with drift depending on a correlated and unobservable stochastic factor $X$ whose dynamics
is a given by a diffusion process.
Then, we study the case where $S$ is a pure jump process whose local characteristics (jump-intensity and jump-size distribution) depend on an
unobservable stochastic factor $X$ given by a Markov jump-diffusion process having common jump times with $S$. This model fits well with high-frequency
data and with the possibility of catastrophic events. Indeed, this kind of events influences both the asset prices and the hidden state variable which
drives their dynamics.
Finally, the last example considers the more general case where the stock price $S$ is a jump-diffusion process and the stochastic factor $X$ is a
correlated Markov jump-diffusion process having common jump times with $S$.
In all these examples, the computation of the optimal value process leads to a filtering problem with respect to the minimal martingale measure $\P^*$ and the historical probability measure $\P$, in presence of jumps.
Filtering problems have been extensively investigated in literature. Results for the case of continuous partially observable systems can be found for instance in~\cite{K, Lis, KO}, the pure jump observations case is analyzed in~\cite{Br, KMM, FR2001, cg09, cg06, c06} and the mixed type observations case
is studied in~\cite{FR2010, fs2012, FSX, ce12, cco1, GrMi, cco2}.
Note that the optimal strategies under partial information for the above mentioned models require the knowledge  of the filter dynamics 
with
respect to $\P^*$, which provides the $\P^*$-conditional law of the stochastic factor $X$ given the information flow on asset prices. Consequently, we derive the filtering equations under $\P^*$ for the three models in Appendix \ref{appendix:a}, by extending the results proved in~\cite{cco1}.

The paper is organized as follows. In Section 2 we describe the financial market model and formulate the hedging problem under partial information according to the local risk-minimization approach. Section 3 is devoted to prove that the underlying price process satisfies the structure condition under the subfiltration $\bH$. The characterization of the optimal strategy, even in presence of jumps, can be found in Section 4. Some Markovian models are discussed in Section 5. Finally, the computation of the filter dynamics for the Markovian models and some proofs are gathered in Appendix.

\section{Hedging problem formulation under partial information}\label{sec:setting}

Let $(\Omega,\F,\P)$ be a probability space endowed with a filtration $\bF:=\{\F_t, \ t\in [0,T]\}$ that satisfies the usual conditions of
right-continuity and completeness, where $T > 0$ is a fixed and finite time horizon; furthermore, we assume that $\F=\F_T$. We consider a simple financial market model where we can find one riskless asset with (discounted) price $1$ and a risky asset whose (discounted) price $S$ is represented
by an $\R$-valued square-integrable c\`{a}dl\`{a}g $(\bF,\P)$-semimartingale satisfying the following structure condition (see e.g.~\cite{s01} for further details):
\begin{equation} \label{eq:SC}
S_t = S_0 + M_t + \int_0^t\alpha_u^\F  \ud \langle M \rangle_u,\quad t \in [0,T],
\end{equation}
where $S_0 \in L^2(\F_0,\P)$\footnote{The space $L^2(\F_0,\P)$ denotes the set of all $\F_0$-measurable random variables $H$ such that $\esp{|H|^2} = \int_\Omega |H|^2\ud {\P} < \infty$.}, $M=\{M_t,\ t \in [0,T]\}$ is an $\R$-valued square-integrable (c\`{a}dl\`{a}g) $(\bF,\P)$-martingale
starting at null, $\langle M\rangle=\{\langle M,M\rangle_t,\ t \in [0,T]\}$ denotes its $\bF$-predictable quadratic variation process and
$\alpha^\F=\{\alpha_t^\F, \ t \in [0,T]\}$ is an $\R$-valued $\bF$-predictable process such that $\int_0^T\left(\alpha_s^\F\right)^2 \ud \langle
M\rangle_s < \infty$ $\P$-a.s..

\begin{remark} \label{rem:structure}
It is quite natural to assume that $S$ is a semimartingale under the real-word probability measure $\P$. Indeed, this is implied by the existence of an equivalent martingale measure, and equivalently by the absence of arbitrage opportunities.
Moreover, according to the results proved in~\cite[page 24]{AS} and~\cite[Theorem 1]{ms95}, if in addition, $S$ has continuous trajectories or c\`adl\`ag paths and the following condition
holds:
$$
\esp{\sup_{t \in [0,T]}S_t^2} < \infty,
$$
then, $S$ satisfies the structure condition with respect to $\bF$ given in  \eqref{eq:SC}.
\end{remark}

Without further mention, all subsequently appearing quantities will be expressed in discounted units.
At any time $t \in [0,T]$, market participants can trade in order to reallocate their wealth. We assume that they have a limitative knowledge on the
market, then their choices cannot be based on the full information flow $\bF$. To describe this scenario, we consider the filtration $\bF^S:=\{\F^S_t,
\ t \in [0,T]\}$ generated by the risky asset price process $S$, i.e. $\F^S_t=\sigma\{S_u, \ 0 \leq  u \leq t \leq T\}$, and the filtration
$\bH:=\{\H_t, \ t \in [0,T]\}$, representing the available information to traders;
both filtrations are supposed to satisfy the usual hypotheses of completeness and right-continuity, and
since the information on asset prices is announced to the public, it is reasonable to assume that the stock price process $S$ is adapted to both
filtrations $\bF$ and $\bH$, that is
\begin{equation}\label{hp:filtration}
\F^S_t \subseteq \H_t \subseteq \F_t, \quad t \in [0,T].
\end{equation}
Condition \eqref{hp:filtration} implies that agents can observe at least the market prices of negotiated assets.

In this market we consider a European-type contingent claim whose final payoff is given by an $\H_T$-measurable random variable $\xi$ such
that  $\esp{|\xi|^2}<\infty$ (or equivalently, $\xi \in L^2(\H_T, \P)$).

Then, the goal is to study the hedging problem of the given contingent claim $\xi$ in the incomplete market driven by $S$ where there are restrictions
on the available information to traders,
via the local risk-minimization approach (see e.g.~\cite{fs1991},~\cite{s93}
and~\cite{s01}).

It is important to stress that the
risky asset price process $S$ turns out to be an $(\bH,\P)$-semimartingale in virtue of condition
\eqref{hp:filtration} on filtrations. Then it admits a semimartingale decomposition with respect to $\bH$,
i.e.
\begin{equation}\label{eq:Semi}
S_t= S_0+N_t+ R_t, \quad t \in [0,T],
\end{equation}
where $N=\{N_t,\ t \in [0,T]\}$ is  an $\R$-valued square-integrable $(\bH, \P)$-martingale with $N_0=0$ and $R=\{R_t,\ t \in [0,T]\}$ is an
$\R$-valued $\bH$-predictable process of finite variation with $R_0=0$. Moreover, since $R$ is $\bH$-predictable this decomposition is unique (see
e.g.~\cite[Chapter III, Theorem 34]{pp}) and will be called the {\em canonical $\bH$-decomposition} of $S$.

On the other hand, the payoff of a given contingent claim is always supposed to be an $\H_T$-measurable random
variable. We observe that all the processes involved are then $\bH$-adapted, and this allows to reduce the hedging problem under partial information to an equivalent one in the case of full information.

We now briefly recall the main concepts and results about the local risk-minimization approach (with respect to $\bH$).

Since we work with both the decompositions of $S$, in the sequel we refer to $M$ as the $\bF$-martingale part of $S$, and $N$ as the $\bH$-martingale part of $S$.

Firstly, we introduce the definition of (hedging) strategy and assume some minimal requirements to make it admissible.
\begin{definition}\label{thetaH-F}
The space $\Theta(\bH)$ (respectively $\Theta(\bF)$)
consists of all $\R$-valued $\bH$-predictable (respectively $\bF$-predictable)
processes $\theta=\{\theta_t,\ t \in [0,T]\}$ satisfying
the following integrability condition:
\begin{equation*}\label{admissible}
\esp{\int_0^T\theta_u^2 \ud \langle N \rangle_u+\left(\int_0^T|\theta_u \ud R_u| \right)^2}<\infty \quad\left( \mbox{resp. } \esp{\int_0^T\theta_u^2 \ud \langle M \rangle_u+\left(\int_0^T|\theta_u| |\alpha^\F_u| \ud \langle M \rangle_u\right)^2}<\infty\right).
\end{equation*}
\end{definition}

\begin{definition}
An $\bH$-{\em admissible strategy} is a pair $\psi=(\theta,\eta)$, where $\theta \in \Theta(\bH)$  and $\eta=\{\eta_t, \ t \in [0,T]\}$ is an
$\R$-valued $\bH$-adapted process
such that the value process
$V(\psi)=\{V_t(\psi), \ t \in [0,T]\}:=\theta S + \eta$ is right-continuous and  square-integrable, i.e. $V_t(\psi) \in L^2(\H_t,\P)$, for each $t
\in [0,T]$.
\end{definition}
Note that $\theta$ and $\eta$ describe the amount of wealth invested in the risky asset and in the riskless asset respectively.

For any $\bH$-admissible strategy $\psi$, we can define the associated {\em cost process} $C(\psi)=\{C_t(\psi), \ t \in [0,T] \}$ which is the
$\R$-valued $\bH$-adapted process given by
$$
C_t(\psi)=V_t(\psi)-\int_0^t \theta_u \ud S_u,
$$
for every $t \in [0,T]$.

In our framework the market is incomplete, then perfect replication of a given contingent claim by a self-financing $\bH$-admissible
strategy is not guaranteed.

However, even if $\bH$-admissible strategies $\psi$ with $V_T(\psi)=\xi$ will in general not be self-financing, it turns out that good $\bH$-admissible
strategies are still self-financing on average in the following sense.

\begin{definition}
An $\bH$-admissible strategy $\psi$ is called {\em mean-self-financing} if the associated cost process $C(\psi)$ is an $(\bH,\P)$-martingale.
\end{definition}

Similarly to \cite{s01}, we introduce the concept of {\it pseudo optimal strategy}.

\begin{definition}\label{def:optimalstrategy}
Let $\xi\in L^2(\H_T, \P)$ be a contingent claim. An $\bH$-admissible strategy $\psi$ such that $V_T(\psi)=\xi$ $\P-\mbox{a.s.}$ is called $\bH$-{\em
pseudo optimal} for $\xi$ if and only if $\psi$ is mean-self-financing and the $(\bH,\P)$-martingale $C(\psi)$ is  strongly orthogonal to the
$\bH$-martingale part, $N$, of $S$, see \eqref{eq:Semi}.
\end{definition}

We have skipped the original definition of {\it locally risk-minimizing strategy}, given in~\cite{s93}, since it is rather technical and delicate.   Moreover, in the one-dimensional case, under mild assumptions on the semimartingale $S$, locally risk minimizing and pseudo optimal strategies coincide, see~\cite[Theorem 3.3]{s01}. The advantage of working with pseudo-optimal strategies is that they can be characterized through an appropriate decomposition of the contingent claim $\xi$.

\begin{definition}
Let $\xi \in L^2(\H_T, \P)$ be the payoff of European-type contingent claim. We say that $\xi$ admits the F\"{o}llmer-Schweizer decomposition with respect to $S$ and $\bH$,  if there exists a random variable $U_0 \in L^2(\H_0, \P)$, a process $\beta^\H\in \Theta(\bH)$ and a square-integrable $(\bH,\P)$-martingale $A=\{A_t,\ t \in [0,T]\}$  with $A_0=0$
strongly orthogonal to the $\bH$-martingale part of $S$, $N$,  such that
\begin{equation}\label{eq:fsdecomposition}
\xi = U_0 + \int_0^T \beta^\H_t \ud S_t + A_T \quad \P-\mbox{a.s.}.
\end{equation}
\end{definition}

\begin{remark}
Some classes of sufficient conditions for the existence of the F\"{o}llmer-Schweizer decomposition are given for example in \cite{ms94, s95, ms95,cvv2010,ccr2}.
\end{remark}

The following result enables us to characterize the $\bH$-pseudo optimal strategy via the F\"{o}llmer-Schweizer decomposition.

\begin{proposition}\label{prop:strategy-FS}
A contingent claim $\xi \in L^2(\H_T, \P)$ admits a unique  $\bH$-pseudo optimal strategy $\psi^*=(\theta^*,\eta^*)$ with $V_T(\psi^*)=\xi$
$\P-\mbox{a.s.}$ if and only if decomposition \eqref{eq:fsdecomposition} holds. The strategy $\psi^*$ is explicitly given by
\[
\theta_t^*=\beta^\H_t, \quad t \in [0,T],
\]
with minimal cost
\[
C_t(\psi^*)= U_0+ A_t, \quad t \in [0,T];
\]
its value process is
\begin{equation*}\label{eq:optimalport}
V_t(\psi^*)=\condesph{\xi-\int_t^T \beta^\H_u \ud S_u}= U_0 + \int_0^t \beta^\H_u \ud S_u + A_t, \quad t \in [0,T],
\end{equation*}
so that $\eta_t^*=V_t(\psi^*)-\beta^\H_t S_t$, for every $t \in [0,T]$.
\end{proposition}

\begin{proof}
For the proof see \cite[Proposition 3.4]{s01}.
\end{proof}

In view of~\cite[Theorem 3.3]{s01} and Proposition \ref{prop:strategy-FS}, finding the F\"ollmer-Schweizer decomposition of a given contingent claim
$\xi$ is important because it allows one to obtain the $\bH$-pseudo optimal strategy. The problem is then how to compute
such a decomposition. If the stock price process $S$ is continuous, the optimal strategy can be calculated
by switching to a particular martingale measure $\P^*$, the so-called {\em minimal martingale measure} (in short MMM), and computing the
Galtchouk-Kunita-Watanabe decomposition of $\xi$ with respect to $S$ under $\P^*$. However, concerning the more general case, that is, when $S$ is
only
c\`{a}dl\`{a}g, there are few results in literature, even under complete information, see e.g.~\cite{cvv2010}. A semimartingale market model under
restricted information has been investigated only in~\cite{ccr2}, as far as we are aware.

\section{Structure condition of the stock price $S$ with respect to $\bH$}\label{sec:SC}

In the sequel we will use the notation ${}^o X$ (respectively, ${}^p X$) to indicate the optional (respectively, predictable) projection with respect
to $\bH$ under $\P$ of a given process $X=\{X_t,\ t \in [0,T]\}$ satisfying $\esp{|X_t|} < \infty$ for every $t \in [0,T]$, defined as the unique
$\bH$-optional (respectively, $\bH$-predictable) process such that  ${}^o X_\tau = \esp{X_{\tau} | \H_{\tau}}$ $\P$-a.s. on $\{\tau < \infty\}$ for
every $\bH$-stopping time $\tau$ (respectively, ${}^p X_\tau = \esp{X_{\tau}| \H_{\tau^-}}$ $\P$-a.s. on $\{\tau < \infty\}$ for every
$\bH$-predictable stopping time $\tau$).

We also denote by $B^{p,\bH}$ the $(\bH,\P)$-predictable dual projection of an $\R$-valued c\`adl\`ag $\bF$-adapted process $B=\{B_t,\ t \in [0,T]\}$ of integrable
variation, defined as
the unique $\R$-valued $\bH$-predictable process $B^{p,\bH}=\{B_t^{p,\bH},\ t \in [0,T]\}$ of integrable variation, such that
\begin{equation*}\label{eq:hpredvar}
\esp{\int_0^T\varphi_t\ud B_t^{p,\bH}}=\esp{\int_0^T\varphi_t\ud B_t},
\end{equation*}
for every $\R$-valued $\bH$-predictable (bounded) process $\varphi=\{\varphi_t, \ t \in [0,T]\}$. See e.g. Section 4.1 of~\cite{ccr1} for further details.

When the risky asset price process $S$ has continuous trajectories, the classical decomposition of $S$ with respect to the filtration $\bH$ has the
form (see, e.g.~\cite{kxy2006} or~\cite{ms2010}):
\begin{equation*}\label{mania}
S_t= S_0+N_t+ \int_0^t {}^p \alpha_u^\F \ud \langle N \rangle_u, \quad t \in [0,T],
\end{equation*}
where the process $N=\{N_t,\ t \in [0,T]\}$ given by
 $$
N_t = M_t + \int_0^t [\alpha_u^\F - {}^p\alpha_u^\F] \ud \langle M \rangle_u, \quad t \in [0,T],
$$
 is an $(\bH, \P)$-martingale. Recall that $M$ denotes the martingale part of $S$ under $\bF$, see \eqref{eq:SC}. Since the quadratic variation process
 $[S]$ of $S$ is defined by
 $$
 [S]_t=S_t^2-2\int_0^t S_{u^-}\ud S_u,\quad t \in [0,T],
 $$
 it turns out to be $\bF^S$-adapted, while in general the predictable quadratic variation $\langle S\rangle$ of $S$ depends on the choice of the
 filtration. Clearly, if $S$ is continuous, we have that
${}^{\bH} \langle N \rangle = {}^{\bF} \langle M \rangle$  and these sharp brackets are $\bF^S$-predictable. Here, the notations ${}^{\bH} \langle
\cdot \rangle$ and ${}^{\bF} \langle \cdot\rangle$
just stress the fact that the predictable quadratic variations are computed with respect to the filtrations $\bH$ and $\bF$, respectively. However, if
it does not create ambiguity, we will always write $\langle M\rangle = {}^{\bF} \langle M \rangle$ and $\langle N\rangle = {}^{\bH} \langle N \rangle$
to simplify the notation.

In presence of jumps these relations are no longer true, since $ {}^{\bF} \langle M^d\rangle  \neq   {}^{\bH}\langle N^d\rangle $, where $M^d$ and
$N^d$ denote the discontinuous parts of the martingales $M$ and $N$, respectively. To compute explicitly the predictable quadratic variations,  we
introduce the integer-valued random measure associated to the jumps of $S$:
\begin{equation*}
m(\ud t,\ud z) = \sum_{s: \Delta S_s \neq 0} \delta_{(s, \Delta S_s)}(\ud t,\ud z),
\end{equation*}
where $\delta_a$ denotes the Dirac measure at  point $a$.

Denote by $\nu^\bF(\ud t, \ud z)$ and $\nu^\bH(\ud t,\ud z)$ the predictable dual projections of $m(\ud t,\ud z)$ under $\P$ with respect to $\bF$ and
$\bH$ respectively (we refer the reader to~\cite{j79} or~\cite{js} for the definition). Then, by~\cite[Chapter II, Corollary 2.38]{js} we get the
following representations of the martingales $M$ and $N$:
$$
M_t = M^c_t + \int_0^t\int_{\R} z(m(\ud t,\ud z) - \nu^\bF(\ud t,\ud z)), \quad t \in [0,T],
$$
$$
N_t = N^c_t + \int_0^t\int_{\R} z(m(\ud t,\ud z) - \nu^\bH(\ud t,\ud z)), \quad t \in [0,T],
$$
where $M^c$ and $N^c$ denote the continuous parts of $M$ and $N$ respectively, and we have $\langle M^c\rangle= \langle N^c\rangle$ as just observed
before.
Hence
$$
\langle M\rangle_t = \langle M^c\rangle_t +  \int_0^t\int_{\R} z^2\nu^\bF(\ud t,\ud z), \quad t \in [0,T],
$$
$$
\langle N\rangle _t =  \langle M^c\rangle _t + \int_0^t\int_{\R} z^2\nu^\bH(\ud t,\ud z), \quad t \in [0,T].
$$
Now, we are in the position to derive the structure condition of $S$ with respect to the filtration $\bH$.
\begin{proposition} \label{prop:N}
Assume that \begin{equation} \label{eq:int_cond}
\esp{\int_0^T\left(\alpha_u^\F\right)^2 \ud \langle M \rangle_u}<\infty.
\end{equation}
Then the $(\bF,\P)$-semimartingale $S$ satisfies the structure condition with respect to $\bH$, i.e.
\begin{equation*} \label{eq:SC_H}
S_t = S_0 + N_t + \int_0^t \alpha_s^\H \ud  \langle N \rangle_s, \quad t \in [0,T],
\end{equation*}
where $\langle N \rangle$ coincides with the $(\bH,\P)$-predictable dual projection of $ \langle M \rangle$, that is, $\langle N \rangle = \langle
M \rangle^{p,\bH}$ and the $\R$-valued $\bH$-predictable process $\alpha^\H=\{\alpha_t^\H,\ t \in [0,T]\}$ given by
\begin{equation*} \label{sc1}
\alpha_t^\H:=\frac{\ud \left(\int_0^t \alpha_u^\F \ud \langle M \rangle_u \right)^{p,\bH}}{\ud   \langle M \rangle_t^{p,\bH}}, \quad t \in [0,T],
 \end{equation*}
satisfies an integrability condition analogous to \eqref{eq:int_cond}.

\end{proposition}

\begin{proof}
By~\cite[Proposition 9.24]{j79} we get that the process $R=\{R_t,\ t \in [0,T]\}$ in decomposition  (\ref{eq:Semi}) is given by
\begin{equation*} \label{eq:R}
R_t  = \left( \int_0^t \alpha_s^\F \ud \langle M \rangle_s \right)^{p,\bH}, \quad t \in [0,T].
\end{equation*}
Now, by applying~\cite[Proposition 4.9]{ccr1} we deduce that $R$ is absolutely continuous with respect to
$\langle M \rangle^{p,\bH}$ and as a consequence, it can be written as
$R_t  = \int_0^t \alpha_s^\H \ud  \langle M \rangle^{p,\bH}_s$, $t \in [0,T]$, where  the process $\alpha^\H$  is the Radon-Nikodym derivative of  $\left(\int \alpha_t^\F \ud \langle M \rangle_t \right)^{p,\bH}$ with respect to $\langle M \rangle^{p,\bH}$.

To prove that $\langle N \rangle = \langle M \rangle^{p,\bH}$ we notice that $\langle M^c\rangle= \langle N^c\rangle$, which is
$\bH$-predictable, and then we only need to show that  $\langle N^d \rangle = \langle M^d \rangle^{p,\bH}$, that is
$$
\int_0^t \int_{\R} z^2\nu^\bH(\ud s,\ud z) = \left( \int_0^t \int_{\R} z^2 \nu^\bF(\ud s,\ud z)\right)^{p,\bH}, \quad t \in [0,T].
$$
To this aim, we observe that by definitions of $\nu^\bF(\ud t,\ud z)$ and $\nu^\bH(\ud t,\ud z)$, for every $\bH$-predictable (bounded) process
$\varphi=\{\varphi_t, \ t \in [0,T]\}$ we have
\begin{align*}
\esp{\int_0^T  \varphi_s \int_{\R} z^2\nu^\bF(\ud s,\ud z)} & = \esp{\int_0^T \int_{\R} \varphi_s  z^2 m (\ud s,\ud z)} = \esp{\int_0^T \int_{\R}
\varphi_s  z^2\nu^\bH (\ud s,\ud z)} \\
& = \esp{\int_0^T  \varphi_s \int_{\R} z^2\nu^\bH(\ud s,\ud z)}.
\end{align*}

Finally, it remains to check that $\alpha^\H$ satisfies the required integrability condition, i.e.
$$
\esp{\int_0^T \left(\alpha_u^\H\right)^2 \ud \langle N \rangle_u} < \infty.
$$
Since for every $\bH$-predictable process $\varphi$ we have
$$
\esp{\int_0^T \varphi_u\alpha_u^\H \ud \langle M\rangle_u}=
\esp{\int_0^T \varphi_u\alpha_u^\H \ud \langle N\rangle_u}=
\esp{\int_0^T \varphi_u (\alpha_u^\F \ud \langle M\rangle_u)^{p,\bH}}=\esp{\int_0^T \varphi_u \alpha_u^\F \ud \langle M\rangle_u},
$$
by choosing $\varphi=\alpha^\H$ and applying the Cauchy-Schwarz inequality, we have
$$
\esp{\int_0^T\left(\alpha_u^\H\right)^2 \ud \langle N \rangle_u} \leq \esp{\int_0^T\left(\alpha_u^\F\right)^2 \ud \langle M \rangle_u} < \infty.
$$
\end{proof}
We conclude the section by considering the case where the $\bF$-predictable quadratic variation of the $(\bF,\P)$-martingale $M$ is absolutely
continuous with respect to the Lebesgue measure,
that is, $\ \langle M\rangle _t = \int_0^t a_s \ \ud s$, $t \in [0,T]$, for some $\R$-valued $\bF$-predictable process $a=\{a_t,\ t \in [0,T]\}$.\\
In such a case, $\ds\langle N \rangle_t = \langle M \rangle_t^{p,\bH} =  \int_0^t  {}^p a_s \ \ud s$ and $\left( \int_0^t \alpha_s^\F \ \ud \langle M \rangle_s
\right)^{p,\bH} =
\int_0^t  {}^p (\alpha_s^\F \ a_s) \ud s$ for each $t \in [0,T]$; hence
$$
\alpha_t^\H =\frac{  {}^p (\alpha_t^\F a_t) }{ {}^p a_t}\I_{\{ {}^p a_t \neq 0 \}}, \quad t \in [0,T].
$$
Moreover
\begin{equation}\label{eq:N}
N_t = {}^oM_t + \leftidx{^o}{\left(\int_0^t\alpha_u^\F a_u \ud u\right)}- \int_0^t {}^p({\alpha_u^\F a_u}) \ud u
\end{equation}
for every $t \in [0,T]$.

Indeed, consider the structure condition of $S$ with respect to $\bF$ given in \eqref{eq:SC} and project it onto $\H_t$, i.e.
\begin{equation} \label{S:proj}
 S_t = {}^oS_t = S_0 + {}^oM_t + \leftidx{^o}
{\left(\int_0^t  \alpha_u^\F a_u \ud u\right)}, \quad t \in [0,T].
\end{equation}

On the other hand, since $R_t  =  \int_0^t {}^p( \alpha_u^\F \ a_u)\ud u$, we get
$$
S_t = S_0 + N_t  + \int_0^t  {}^p( \alpha_u^\F a_u)  \ud u, \quad t \in [0,T],
$$
hence  \eqref{eq:N} is proved.
\begin{remark}
In particular, the $(\bH, \P)$-martingale $N=\{N_t,\ t \in [0,T]\}$ with $N_0=0$ can be decomposed as the sum of three $(\bH, \P)$-martingales, see
equation \eqref{eq:3_mg} below. Indeed, \eqref{S:proj} can be written as
$$
S_t=S_0 + {}^o M_t + \int_0^t \leftidx{^o}{\left(\alpha_u^\F a_u\right)} \ud u + m_t, \quad t \in [0,T],
$$
where the process $m=\{m_t,\ t \in [0,T]\}$ is given by
\[
m_t:=\leftidx{^o}{\left(\int_0^t\alpha_u^\F a_u \ud u\right)}- \int_0^t \leftidx{^o}{(\alpha_u^\F a_u)} \ud u, \quad t \in [0,T].
\]
Note that the process ${}^oM=\{{}^oM_t,\ t \in [0,T]\}$, as well as the process $m$, is an $(\bH, \P)$-martingale (see e.g.~\cite[Chapter
IV, Theorem T1]{Br} for the proof).
Set now
$$
\widetilde m_t :=\int_0^t {}^o({\alpha_u^\F a_u}) \ud u - \int_0^t  {}^p ({\alpha_u^\F a_u}) \ud u, \quad t \in [0,T].
$$
Then, the process $\widetilde m=\{\widetilde m_t, \ t \in [0,T]\}$
turns out to be an $(\bH, \P)$-martingale.
Hence
\begin{equation} \label{eq:3_mg}
N_t={}^oM_t + m_t+\widetilde m_t,
\end{equation}
for every $t \in [0,T]$.
\end{remark}

We will see how to compute explicitly the structure conditions of $S$ with respect to $\bF$ and $\bH$ in the models discussed in Section \ref{sec:markovian_models}.

\section{The $\bH$-pseudo optimal strategy}\label{sec:optimal_strategy}

In the case of full information, when the semimartingale $S$ has continuous trajectories, it is proved in~\cite[Theorem 3.5]{s01}, that there exists
the $\bH$-pseudo optimal
strategy and that can be obtained via the Galtchouk-Kunita-Watanabe decomposition of the contingent claim $\xi$ with respect to $S$ under the MMM
$\P^*$. This is essentially due to the fact that, in the case of continuous trajectories, the MMM preserves orthogonality, and then the
Galtchouk-Kunita-Watanabe decomposition of the contingent claim under the MMM $\P^*$ provides the F\"{o}llmer-Schweizer decomposition of the contingent
claim under the historical probability measure $\P$. Obviously, this does not work if the $(\bF,\P)$-semimartingale $S$ exhibits jumps.
However, also in presence of jumps, the MMM and the Galtchouk-Kunita-Watanabe decomposition of the contingent claim $\xi$ still represent the key tools to compute the $\bH$-pseudo optimal strategy, see e.g.~\cite{cvv2010} for the semimartingale market model under under full information.\\
Here, we provide a similar criterion to characterize the pseudo optimal strategy in the partial information case, see equation \eqref{eq:betaH}.\\
Furthermore, in the next section we will show some Markovian models affected by an unobservable stochastic factor, where this computation leads to filtering problems under the historical probability measure $\P$ and the MMM $\P^*$.\\
For reader's convenience, firstly we recall the definition of MMM with respect to the filtration $\bF$.
\begin{definition} \label{def:MMM}
An equivalent martingale measure $\P^*$ for $S$ with square-integrable density $\ds \frac{\ud \P^*}{\ud \P}$ is called
 {\em minimal martingale measure} (for $S$) if $\P^*=\P$ on $\F_0$ and if every square-integrable $(\bF, \P)$-martingale, strongly orthogonal to the
 $\bF$-martingale part of $S$, $M$, is also an $(\bF, \P^*)$-martingale.
\end{definition}

If we assume that
  \begin{equation*}\label{eq:hpAS2}
  1-\alpha_t^\F \Delta M_t>0 \quad \P-\mbox{a.s.} \quad \forall \ t \in [0,T],
  \end{equation*}
and
 \begin{equation}\label{eq:square_integrability_L}
\esp{\exp\left\{\frac{1}{2}\int_0^T \left(\alpha_t^\F\right)^2 \ud \langle M^c\rangle_t+ \int_0^T \left(\alpha_t^\F\right)^2 \ud \langle M^d \rangle_t
\right\}} < \infty,
  \end{equation}
 where $M^c$ and $M^d$ denote the continuous and the discontinuous parts of the $(\bF,\P)$-martingale $M$ respectively and $\alpha^\F$ is given in
 \eqref{eq:SC},
  then by the Ansel-Stricker Theorem (see~\cite{AS}) there exists the MMM $\P^*$ for $S$, which is defined thanks to the density process $L=\{L_t,\ t
  \in [0,T]\}$ given by
\begin{equation}\label{eq:MMMpartial}
L_t:=\left.\frac{\ud \P^*}{ \ud \P}\right|_{\F_t}=\doleans{-\int \alpha_u^\F \ud M_u}_t,\quad t \in [0,T],
\end{equation}
where the notation $\mathcal{E}(Y)$ refers to the Dol\'{e}ans-Dade exponential of an $(\bF,\P)$-semimartingale $Y$.

We observe that condition \eqref{eq:square_integrability_L} implies that the nonnegative $(\bF, \P)$-local martingale $L$ is indeed a square-integrable
$(\bF, \P)$-martingale, see e.g.~\cite{ps2008}, and also that \eqref{eq:int_cond} holds true.

Now, we assume that
  \begin{equation*}\label{eq:hpASH1}
  1-\alpha_t^\H \Delta N_t>0 \quad \P-\mbox{a.s.} \quad \forall \ t \in [0,T],
  \end{equation*}
and
 \begin{equation*}\label{eq:square_integrability_L0}
\esp{\exp\left\{\frac{1}{2}\int_0^T \left(\alpha_t^\H\right)^2 \ud \langle N^c\rangle_t+ \int_0^T \left(\alpha_t^\H\right)^2 \ud \langle N^d \rangle_t
\right\}} < \infty,
  \end{equation*}
 where, as usual $N^c$ and $N^d$ denote the continuous and the discontinuous parts of the $(\bH,\P)$-martingale $N$ respectively.
 Then similarly to before, we define $\P^0$ as the probability measure on $(\H_T, \Omega)$  such that

\begin{equation}\label{eq:MMMpartial2}
L^0_t:=\left.\frac{\ud \P^0}{\ud \P}\right|_{\H_t}=\doleans{-\int \alpha^\H_u \ud N_u}_t, \quad \forall \ t \in [0,T].
\end{equation}
We notice that $L^0$ is a square-integrable $(\bH, \P)$-martingale and $\P^0$, equivalent to $\P$ over the filtration $\bH$, provides the {\em MMM with respect to the filtration $\bH$}.

We are now in the position to state the following result.

\begin{proposition}\label{prop:strategia_ottima}
Let $\xi \in L^2(\H_T, \P)$ be a contingent claim that admits the F\"{o}llmer-Schweizer decomposition with respect to $\bH$ and $S$, and $\psi^*=(\theta^*,\eta^*)$  be the associated $\bH$-pseudo optimal strategy.
Then, the optimal value process $V(\psi^*)=\{V_t(\psi^*),\ t \in [0,T]\}$ is given by
\begin{equation*} \label{eq:portfolio_value}
V_t(\psi^*)= \bE^{\P^0}\left[\xi | \H_t\right], \quad t \in [0,T],
\end{equation*}
where $\bE^{\P^0}\left[\cdot | \H_t\right]$ denotes the conditional expectation with respect to $\H_t$ computed under $ \P^0$; moreover, the first
component $\theta^*$ of the $\bH$-pseudo optimal strategy $\psi^*$
 is given by
\begin{equation} \label{eq:optimal_strategy}
\theta_t^*=\frac{ \ud  {}^\bH \langle V^m(\psi^*), N\rangle_t}{ \ud   {}^\bH\langle N \rangle_t}, \quad  t \in [0,T] ,
\end{equation}
where $V^m(\psi^*)$ is the $(\bH,\P)$-martingale part of the process $V(\psi^*)$ and here the sharp brackets are computed under $\P$.
\end{proposition}

\begin{proof}
Since $L^0$ given in \eqref{eq:MMMpartial2} is a square-integrable $(\bH, \P)$-martingale, by Cauchy-Schwarz inequality we get that
\[
 \bE^{ \P^0}[|\xi |] = \esp{|\xi|  L^0_T} \leq \esp{\xi^2}^{1/2} \esp{(L^0_T)^2}^{1/2}<\infty,
\]
which means that $\xi \in L^1(\H_T, \P^0)$. \\
Consider the F\"{o}llmer-Schweizer decomposition of $\xi$ with respect to $S$ and $\bH$, see \eqref{eq:fsdecomposition}, and let
$\psi^*=(\theta^*,\eta^*)$ be the $\bH$-pseudo optimal strategy. Then, by Proposition \ref{prop:strategy-FS} we get $\theta^*=\beta^\H$ and the optimal
value process $V(\psi^*)$ satisfies
\begin{equation*}\label{eq:decomp_portfolio}
V_t(\psi^*)=  U_0 + \int_0^t \beta^\H_u \ud S_u + A_t, \quad t \in [0,T].
\end{equation*}
Observe that $\int \beta^\H_t \ud S_t$ is an $(\bH, \P^0)$-martingale since $\int \beta^\H_t \ud N_t$ and $ L$ are $(\bH, \P)$-martingales (see the
proof of Theorem 3.14 in~\cite{fs1991}) and $A$ turns out to be an $(\bH, \P^0)$-martingale  by definition of the MMM with respect to the filtration $\bH$. Then, the optimal value process
$V(\psi^*)$ is an $(\bH,\P^0)$-martingale, and as a consequence it can be written as
\[
V_t(\psi^*)= \bE^{\P^0}\left[V_T(\psi^*) | \H_t\right] =  \bE^{\P^0}\left[\xi | \H_t\right], \quad t \in [0,T].
\]
Finally, to compute the $\bH$-pseudo optimal strategy we consider the $(\bH, \P)$-martingale part of the process $V(\psi^*)$ given by
\[
V_t^m(\psi^*)=  U_0 + \int_0^t \beta^\H_u \ud N_u + A_t, \quad t \in [0,T].
\]
Then, taking the predictable quadratic covariation
with respect to the $\bH$-martingale part $N$ of $S$ computed under $\P$ into account, we get that
\[
\ud {}^ \bH \langle V^m(\psi^*), N \rangle_t =  \beta^\H_t \ud  {}^ \bH \langle N \rangle_t, \quad t \in [0,T],
\]
since $A$ is strongly orthogonal to $N$  under $\P$. Then, we obtain equation \eqref{eq:optimal_strategy}.
\end{proof}

When the stock price process $S$ has continuous trajectories, the optimal value process can be characterized in terms of the MMM $\P^*$ with respect to the filtration $\bF$ as proved in Corollary \ref{cor:strategia_ottima} below. We start with a useful lemma.

\begin{lemma}\label{lemma:MMM}
Assume that $S$ has continuous trajectories. Then the MMM $\P^0$ with respect to the filtration $\bH$ coincides with the restriction on the filtration $\bH$ of the MMM $\P^*$ with respect to the filtration $\bF$.
\end{lemma}

\begin{proof}
The proof is postponed to Appendix \ref{appendix:b}.
\end{proof}

\begin{corollary}\label{cor:strategia_ottima}
Assume that $S$ has continuous trajectories ad let $\xi \in L^2(\H_T, \P)$ be a contingent claim that admits the F\"{o}llmer-Schweizer decomposition with respect to $\bH$ and $S$, and $\psi^*=(\theta^*,\eta^*)$  be the associated $\bH$-pseudo optimal strategy.
Then, the optimal value process $V(\psi^*)=\{V_t(\psi^*),\ t \in [0,T]\}$ is given by
\begin{equation*} \label{eq:portfolio_value_cont}
V_t(\psi^*)= \bE^{\P^*}\left[\xi | \H_t\right], \quad t \in [0,T],
\end{equation*}
where $\bE^{\P^*}\left[\cdot | \H_t\right]$ denotes the conditional expectation with respect to $\H_t$ computed under $ \P^*$; moreover, the first
component $\theta^*$ of the $\bH$-pseudo optimal strategy $\psi^*$
 is given by
\begin{equation} \label{eq:optimal_strategy_cont}
\theta_t^*=\frac{ \ud  {}^\bH \langle V(\psi^*), S\rangle_t}{ \ud   {}^\bH\langle S \rangle_t}, \quad  t \in [0,T] ,
\end{equation}
where the sharp brackets are computed under $\P$.
\end{corollary}

\begin{proof}
The proof follows by Proposition \ref{prop:strategia_ottima} observing that, in virtue of Lemma \ref{lemma:MMM}, the optimal value process
$V(\psi^*)$  can be written as
\[
V_t(\psi^*)= \bE^{\P^0}\left[\xi | \H_t\right]=  \bE^{\P^*}\left[\xi | \H_t\right], \quad t \in [0,T].
\]
Finally, since the finite variation part of $S$ is continuous we get that $\ud  {}^ \bH \langle N \rangle=\ud  {}^ \bH \langle S \rangle$ and $\ud {}^ \bH \langle V^m(\psi^*), N \rangle=\ud {}^ \bH \langle V(\psi^*), S \rangle$, which leads to \eqref{eq:optimal_strategy_cont}.
\end{proof}

Clearly Corollary \ref{cor:strategia_ottima} furnishes a characterization of the $\bH$-pseudo-optimal strategy $\beta^\H$ in terms of the MMM $\P^*$ with respect to $\bF$ that holds when $S$ has continuous trajectories. When $S$ exhibits jumps it is not possible to provide an analogous characterization of the optimal value process. This is essentially due to the fact that in general the MMM $\P^0$ with respect to the filtration $\bH$ does not coincide with the restriction of $\P^*$ over $\bH$. Then, to compute explicitly the $\bH$-pseudo-optimal strategy we follow the approach suggested by \cite{cvv2010} in the full information context.

Assume that $\xi$ admits the F\"{o}llmer-Schweizer decomposition of $\xi$ with respect to $S$ and $\bF$,  i.e.
\begin{equation}\label{eq:fsdecompositionF}
\xi = \widetilde U_0 + \int_0^T \beta_t^\F \ud S_t + \widetilde A_T \quad \P-\mbox{a.s.},
\end{equation}
where $U_0 \in L^2(\F_0, \P)$, $\beta^\F \in \Theta(\bF)$ and $\widetilde A=\{\widetilde A_t,\ t \in [0,T]\}$ is a square-integrable
$(\bF,\P)$-martingale with $\widetilde A_0=0$ strongly orthogonal to the $\bF$-martingale part $M$ of $S$ under $\P$.

By applying Proposition \ref{prop:strategy-FS} with the choice $\bH=\bF$, we know that $\beta^\F$ provides the pseudo-optimal strategy under full information.

In the sequel we provide a characterization of the $\bH$-pseudo-optimal strategy $\beta^\H$ and discuss the relationship between $\beta^\H$ and $\beta^\F$.

Denote by $\Theta(\bF,\P^*)$ ($\Theta(\bH,\P^*)$, respectively)
the set of all $\R$-valued $\bF$-predictable (respectively, $\bH$-predictable)
processes $\delta=\{\delta_t,\ t \in [0,T]\}$ satisfying
the following integrability condition:
\begin{equation*}
\espp{\int_0^T\delta_u^2 \ud \langle S\rangle_u}<\infty.
\end{equation*}

In the rest of the section we assume $\xi$ to be square-integrable with respect to  $\P^*$.

Let us observe that since $S$ is a $\P^*$-martingale with respect to both the filtrations $\bF$ and $\bH$, the random variable $\xi$ admits
Galtchouk-Kunita-Watanabe decomposition with respect to $S$ and both the filtrations $\bF$ and $\bH$ under $\P^*$, i.e.
\begin{equation}\label{GKW1}
\xi =  \widetilde U_0 + \int_0^T\widetilde \beta_u^\F \ud S_u + \widetilde G_T \quad \P^*-\mbox{a.s.},
\end{equation}

\begin{equation}\label{GKW2}
\xi =  U_0 + \int_0^T \widetilde \beta^\H_u \ud S_u +G_T \quad \P^*-\mbox{a.s.},
\end{equation}
where $\widetilde U_0 \in L^2(\F_0, \P^*)$, $U_0 \in L^2(\H_0, \P^*)$, $\widetilde \beta^\F\in \Theta(\bF,\P^*)$, $\widetilde \beta^\H\in
\Theta(\bH,\P^*)$, $\widetilde G=\{\widetilde G_t, \ t \in [0,T]\}$ and $G=\{G_t, \ t \in [0,T]\}$  are square-integrable $(\bF, \P^*)$ and  $(\bH,\P^*)$-martingales respectively with $\widetilde G_0=G_0=0$, strongly orthogonal to $S$ under $\P^*$.  \\
On the other hand, if $S$ turns out to be also square-integrable with respect to $\P^*$,
 the $\P^*$-martingale property of $S$ with respect to both the filtrations $\bF$ and $\bH$ also ensures that we can apply~\cite[Theorem 3.2]{ccr1} which
 provides the
Galtchouk-Kunita-Watanabe decomposition of a square-integrable random variable under partial information with respect to $\P^*$. More precisely, every
$\xi \in L^2(\F_T,\P^*)$ can be uniquely written as
\begin{equation}\label{eq:GKWpartial}
\xi =  U_0^{'} + \int_0^T H^\H_u \ud S_u + G_T^{'} \quad \P^*-\mbox{a.s.},
\end{equation}
where $U^{'}_0 \in L^2(\F_0,\P^*)$, $H^\H=\{H_t^\H,\ t \in [0,T]\} \in \Theta(\bH,\P^*)$ and $G^{' }=\{G_t^{'},\ t \in [0,T]\}$  is a square-integrable
$(\bF, \P^*)$-martingale with $G_0^{'}=0$ weakly orthogonal\footnote{We say that a square-integrable $(\bF,\P)$-martingale $O$ is {\em weakly
orthogonal} to a square-integrable $(\bF,\P)$-martingale $M$ if the following condition
\begin{equation*} \label{eq:orthogcond}
\esp{O_T \int_0^T \varphi_t \ud M_t}=0,
\end{equation*}
holds for all processes $\varphi \in \Theta(\bH)$.} to $S$ under $\P^*$, according to Definition 2.1 given in \cite{ccr1}.

\begin{lemma}\label{l1}
Assume $\xi \in L^2(\H_T,\P^*)$  and that $S$ is square-integrable with respect to $\P^*$. Let $\widetilde \beta^\H\in
\Theta(\bH,\P^*)$ and $H^\H \in \Theta(\bH,\P^*)$ be the integrands in the decompositions \eqref{GKW2} and \eqref{eq:GKWpartial} respectively.
Then
\begin{equation}\label{eQ}
H^\H_t=\widetilde \beta_t^\H,  \quad \forall t \in [0,T].
 \end{equation}
\end{lemma}

\begin{proof}

Let $\xi \in L^2(\H_T,\P^*)$ and consider decomposition \eqref{eq:GKWpartial}. By taking the conditional expectation with respect to $\H_T$ under
$\P^*$, we get
\begin{align}
\xi & =  \condesphTT{U_0^{'}} + \int_0^T H^\H_u \ud S_u + \condesphTT{G_T^{'}}  = \widehat U_0 + \int_0^T H^\H_u \ud S_u + \widehat G_T, \label{eq:GKWpartial1}
\end{align}
where we have set $\widehat U_0:=\condesphoo{U_0^{'}}$ and $\widehat G_t:=\condespHH{G_t^{'}} + \condespHH{U_0^{'}}-\condesphoo{U_0^{'}}$, for every $t
\in [0,T]$, so that $\widehat G=\{\widehat G_t, \ t \in [0,T]\}$ turns out to be a square-integrable $(\bH,\P^*)$-martingale with $\widehat G_0=0$
weakly orthogonal to $S$ under $\P^*$. Indeed, $\condespHH{U_0^{'}}-\condesphoo{U_0^{'}} \in L^2(\H_t,\P^*)$, for every $t \in [0,T]$, is clearly
weakly orthogonal to $S$ under $\P^*$ thanks to the martingale property of $S$ with respect to both the filtrations $\bF$ and $\bH$. Furthermore,
for every $\varphi \in \Theta(\bH)$ we have
$$
\espp{\condesphTT{G_T^{'}}\int_0^T \varphi_u \ud S_u}=\espp{\condesphTT{G_T^{'}\int_0^T \varphi_u \ud S_u}}=\espp{G_T^{'}\int_0^T \varphi_u \ud
S_u}=0,
$$
since $G^{'}$ is weakly orthogonal to $S$ under $\P^*$.
Moreover, $\widehat U_0 \in L^2(\H_0,\P^*)$ and since $\widehat G$ is $\bH$-adapted, it is also strongly orthogonal to $S$ under $\P^*$,
see~\cite[Remark 2.4]{ccr1}. Then, by uniqueness of the Galtchouk-Kunita-Watanabe decomposition, representations \eqref{eq:GKWpartial1} and
\eqref{GKW2} for $\xi$ coincide, and in particular this implies \eqref{eQ}.
\end{proof}

The following proposition provides the relationship between the strategies $\beta^\F$ and $\beta^\H$ in the continuous case.

\begin{proposition}\label{prop:relazione_strategie}
 Let $\xi \in L^2(\H_T,\P^*)$ be a contingent claim and assume that $S$ is continuous and square-integrable with respect to $\P^*$.
Then, the following relationship  between the $\bH$-pseudo optimal strategy $\beta^\H$ and the $\bF$-pseudo
optimal strategy $\beta^\F$ holds
\begin{equation}\label{eq:formula_strategia_caso_continuo}
\beta_t^\H = {}^{p,*}\beta_t^\F, \quad t \in [0,T].
\end{equation}
Here, the notation ${}^{p,*}D$ refers to the $(\bH,\P^*)$-predictable projection of an $\R$-valued integrable process $D=\{D_t,\ t \in [0,T]\}$.
 \end{proposition}

 \begin{proof}
When $S$ has continuous trajectories, decompositions \eqref{eq:fsdecompositionF} (with respect to $\bF$) and \eqref{eq:fsdecomposition} (with respect
to $\bH$) coincide to the corresponding Galtchouk-Kunita-Watanabe decompositions under $\P^*$, see \eqref{GKW1} and \eqref{GKW2} above.
Lemma \ref{l1} implies that $\beta^\H = H^\H$ and since $\langle S \rangle$ is $\bH$-predictable, due to the fact that in this case $\langle S \rangle=[S]$, which is $\bF^S$-adapted by
definition, under $\P^*$, by applying  ~\cite[Proposition 4.1]{ccr1} we get  \eqref{eq:formula_strategia_caso_continuo}.
\end{proof}

In the general case, i.e. when $S$ also exhibits jumps, the relationship between $\beta^\H$ and $\beta^\F$ is more complicated.
In~\cite{cvv2010},
the relationship between  $\widetilde \beta^\F$ and  $\beta^\F$, given in \eqref{GKW1} and \eqref{eq:fsdecompositionF} respectively, is written in
terms of the local characteristics associated to $\widetilde G$ under
$\P^*$. A similar result can be applied to derive the relationship between  $\widetilde \beta^\H$ and  $\beta^\H$, given in \eqref{GKW2} and
\eqref{eq:fsdecomposition} respectively, in terms of the local characteristics associated to $G$ under $\P^*$.

We are now in the position to state the following result.
\begin{proposition}\label{casoJ}
Let $\xi \in L^2(\H_T,\P^*)$ be a contingent claim and assume that $S$ is square-integrable with respect to $\P^*$.
The first component of the $\bH$-pseudo optimal strategy $\psi^*=(\beta^\H,\eta^*)$ is given by

\begin{equation}\label{eq:betaH}
\beta^\H_t=H^\H_t+\phi^\H_t, \quad  t \in [0,T].
\end{equation}
In other terms,
\begin{equation} \label{Com}
\beta^\H_t=\frac{\ud  (\int_0^t \widetilde \beta_u^\F \  \ud \langle S\rangle_u)^{p, \bH, *} }{\ud  \langle S \rangle^{p, \bH, *}_t} + \phi_t^\H =
\frac{\ud  (\int_0^t  \beta_u^\F \ \ud \langle S\rangle_u)^{p, \bH, *} }{\ud  \langle S \rangle^{p, \bH, *}_t} + \phi_t^\H- \frac{\ud  (\int_0^t
\phi_u^\F \ \ud \langle S\rangle_u)^{p, \bH, *} }{\ud  \langle S \rangle^{p, \bH, *}_t},
\end{equation}
for every $t \in [0,T]$, where $D^{p, \bH, *}$ denotes the $(\bH,\P^*)$-predictable dual projection
of an $\R$-valued process $D=\{D_t,\ t \in [0,T]\}$ of finite variation, and
the processes $\phi^\F=\{\phi_t^\F,\ t \in [0,T]\}$ and $\phi^\H=\{\phi_t^\H,\ t \in [0,T]\}$ are respectively given by
\begin{equation} \label{Com2}
\phi^\F_t = \frac{ \ud {}^\bF\langle [\widetilde G, S],  -\int_0^\cdot \alpha^\F_r \ud M_r \rangle_t }
{ \ud {}^\bF\langle S \rangle_t} ,
\quad \phi^\H_t = \frac{ \ud {}^\bH\langle [G, S],  -\int_0^\cdot \alpha^\H_r \ud N_r \rangle_t }
{ \ud {}^\bH\langle S \rangle_t} ,
\end{equation}
for every $t \in [0,T]$, where the sharp brackets are computed under $\P$.
\end{proposition}
\begin{proof}
Taking  Lemma \ref{l1} into account, by~\cite[Proposition 4.9]{ccr1} we obtain
\begin{equation*} \label{com1}
H^\H_t=\widetilde \beta^\H_t = \frac{\ud  (\int_0^t \widetilde \beta_u^\F  \ud \langle S\rangle_u)^{p, \bH, *} }{\ud  \langle S \rangle^{p, \bH, *}_t},\quad t
\in [0,T].
\end{equation*}
Then, by applying~\cite[Theorem 3.2]{cvv2010}, we get $\beta^\H =   \widetilde \beta^\H +  \phi^\H$ and  $\beta^\F =   \widetilde \beta^\F +  \phi^\F$,
and then equalities \eqref{Com}.
Finally, the expressions in \eqref{Com2} follow by~\cite[Remark on page 8]{cvv2010}.
\end{proof}

In the next section we will discuss some Markovian models in presence of an unobservable stochastic factor which affects the stock price dynamics, and
we will show how the computation of the  $\bH$-pseudo optimal strategy leads to filtering problems under the MMM $\P^*$ and the real-world
probability measure $\P$.

\section{Markovian models}\label{sec:markovian_models}
In this section we wish to apply our results to
some Markovian models.
We assume that the dynamics of the risky asset price process $S$ depends  on some unobservable process $X$, which may represent the activity of other
markets, macroeconomics factors or microstructure rules that drive the market.

We consider a European-type contingent claim whose payoff $\xi\in L^2(\H_T,\P)\cap L^2(\H_T,\P^*)$ is of the form
\begin{equation*}\label{eq:claim}
\xi=  H (T, S_T),
\end{equation*}
where $H (t, s)$ is a deterministic function.
We define the processes  $V^\F$ and $V^\H$ by setting
$$
V_t^\F:= \bE^{\P^*}\left[H(T,S_T)|\F_t\right],\quad V_t^\H:= \bE^{\P^*}\left[H(T,S_T)|\H_t\right], \quad t \in [0,T].
$$
If the pair $(X,S)$ is an $(\bF, \P^*)$-Markov process, then there exists a measurable function $g(t,x,s)$ such that
\begin{equation}\label{eq:G}
V^\F_t=\mathbb{E}^{\P^*}\left[H(T,S_T)|\F_t\right]=g(t,X_t,S_t)
 \end{equation}
for every $t \in [0,T]$ and
 \begin{equation}\label{eq:valore_portafoglio}
 V_t^\H= \bE^{\P^*}\left[\bE^{\P^*}\left[ H(T,S_T)|\F_t\right]|\H_t\right] = \mathbb{E}^{\P^*}\left[g(t,X_t,S_t)|\H_t\right], \quad t \in [0,T].
 \end{equation}

We denote by $\L^*_{X,S}$ the $(\bF,\P^*)$-Markov generator of the pair $(X,S)$. Then, by~\cite[Chapter 4, Proposition 1.7]{ek86} the process
\[
\left\{f(t,X_t,S_t)-\int_0^t\L^*_{X,S}f(u,X_u,S_u) \ud u,\quad t \in [0,T]\right\}
\]
is an $(\bF, \P^*)$-martingale for every function $f(t,x,s)$ in the domain of the operator $\L^*_{X,S}$, denoted by $D(\L^*_{X,S})$. Then the following result, which allows to compute the function $g(t,x,s)$, holds.
\begin{lemma}\label{lemma:caratterizzazioneG}
Let $\widetilde{g}(t,x,s)\in D(\L^*_{X,S})$ such that
\begin{equation}\label{eq:problema}
\left\{
\begin{aligned}
\L^*_{X,S} \widetilde{g}(t,x,s)&=0, \quad t \in [0,T)\\
\widetilde{g}(T,x,s)&=H(T,s).
\end{aligned}
\right.
\end{equation}
Then $\widetilde{g}(t,X_t,S_t)= g(t,X_t,S_t)$, for every $t \in [0,T]$, with $g(t,x,s)$ given in \eqref{eq:G}.
\end{lemma}

\begin{proof}
Let $\widetilde{g}(t,x,s)\in D(\L^*_{X,S})$ be the solution of \eqref{eq:problema}. Then the process $\ds \left\{\widetilde{g}(t, X_t, S_t), \ t \in [0,T]\right\}$
is an $(\bF, \P^*)$-martingale and since $\widetilde{g}(T,X_T,S_T)=H(T,S_T)$, by the martingale property we get that $\widetilde{g}(t,X_t,S_t)=\bE^{\P^*}[H(T, S_T)|\F_t]$.
\end{proof}

In the computation of the $\bH$-pseudo optimal strategies we will consider the case where the information available to traders is represented by the filtration generated by the stock price
process $S$; in other terms, we assume that
\begin{equation}\label{hp:filtrazione_esempi}
\H_t=\F^S_t \qquad \forall \ t\in [0,T].
\end{equation}

We define the filter $\pi(f)=\left\{\pi_t(f), \ t \in [0,T]\right\}$, by setting for each $t \in [0,T]$
\begin{equation*} \label{def:filtro}
\pi_t(f) : =  \mathbb{E}^{\P^*} [ f(t,X_t,S_t) | \F^S_t ]
\end{equation*}
for any measurable function $f(t,x,s)$ such that $\espp{|f(t,X_t,S_t)|}< \infty$, for every $t\in [0,T]$.  It is known  that $\pi(f)$ is a probability
measure-valued process with c\`{a}dl\`{a}g trajectories (see~\cite{KO}), which provides the $\P^*$-conditional law of $X$ given the information flow.

Then, by  \eqref{eq:valore_portafoglio} the
process $V^\H$ can be written in terms of the filter as
\begin{equation}\label{eq:VH}
V_t^\H=\pi_t(g) \qquad \forall \ t \in [0,T],
\end{equation}
where the function $g(t,x,s)$ is the solution of the problem with final value \eqref{eq:problema}.\\
Therefore we can characterize the integrand $\widetilde \beta^\H$ in the Galtchouk-Kunita-Watanabe decomposition \eqref{GKW2} of $\xi$ under partial information as
\begin{equation*}\label{eq:strategia_filtro1}
\widetilde\beta^\H_t=H^\H_t=\frac{\ud \langle \pi(g), S\rangle_t^{*,\bH} }{\ud \langle S\rangle_t^{*,\bH} }, \quad t \in [0,T],
\end{equation*}
where $\langle \ \rangle^{*,\bH}$ denotes the sharp bracket computed with respect to $\bH$ and $\P^*$.

Finally by Proposition \ref{casoJ} we get that the first component of the $\bH$-pseudo optimal strategy is given by
\begin{equation}\label{eq:strategia_filtro}
\beta^\H_t=\widetilde \beta^\H_t+\phi^\H_t=\frac{\ud \langle \pi(g), S\rangle_t^{*,\bH} }{\ud \langle S\rangle_t^{*,\bH} }+\frac{\ud {}^\bH\langle [G, S], -\int_0^\cdot \alpha_s^\H \ud N_s \rangle_t}{\ud {}^\bH \langle S \rangle_t}, \quad t \in [0,T],
\end{equation}
where $G$ is the $(\bH, \P^*)$-martingale in the Galtchouk-Kunita-Watanabe decomposition \eqref{GKW2} of $\xi$, given by

\[
G_t=-U_0+\pi_t(g)-\int_0^t\widetilde \beta^\H_u \ud S_u, \quad t \in [0,T].
\]

In the following, we compute explicitly the process $\widetilde \beta^\H$ and also provide the $\bH$-pseudo optimal strategy $\psi^*=(\beta^\H,\eta^*)$ for a diffusion, a pure jump and a jump-diffusion market model,
respectively, by characterizing the process $\phi^\H$.

\subsection{A diffusion market model}\label{sec:cont_model}

In the first model we consider the case where the dynamics of the risky asset price process $S$ is a geometric diffusion process which depends on an unobservable stochastic factor $X$ given by a Markovian diffusion process, correlated with $S$. Precisely we assume that the pair $(X,S)$ satisfies the following system of stochastic differential equations (in short SDEs):
\begin{equation}\label{eq:sistema1}
\left\{
\begin{aligned}
\ud X_t&= \mu_0(t, X_t) \ud t + \sigma_0(t,X_t) \ud W^0_t , \quad X_0=x\in \R, \\
\ud S_t&=S_{t}\left(\mu_1(t, X_t, S_t)\ud t+\sigma_1(t,S_t)\ud W^1_t\right),\quad S_0=s>0,
\end{aligned}
\right.
\end{equation}
where  $W^0=\{W^0_t, \ t \in [0,T]\}$ and $W^1=\{W^1_t, \ t \in [0,T]\}$ are $(\bF, \P)$-Brownian motions such that $\langle W^0, W^1\rangle_t = \rho t$, for every $t \in [0,T]$, with
$\rho \in [-1,1]$, the coefficients $\mu_0(t,x)$, $ \sigma_0(t,x)>0$, $\mu_1(t,x,s)$ and $\sigma_1(t,s)>0$  are $\R$-valued measurable functions of their arguments. For simplicity we take:
\begin{equation}\label{ass:boundedness_1}
\mu_1(t, X_t, S_t)< c_1 \ \mbox{ and } \ 0<c_2<\sigma_1(t, S_t)<c_3, \ \forall t \in [0,T]
\end{equation}
for some constants, $c_1,c_2,c_3$.

We assume that a unique strong solution for the system \eqref{eq:sistema1} exists, see for instance~\cite{GS}. In particular, this implies
that the pair $(X, S)$ is an $(\bF, \P)$-Markov process.

\subsubsection{Structure conditions of the stock price $S$ with respect to $\bF$ and $\bH$.}

By \eqref{ass:boundedness_1}, and $S_t$ and $\sigma_1(t,S_t)>0$ for every $t \in [0,T]$, we get that $S$ satisfies the structure condition with respect to $\bF$, i.e.
\[
S_t=S_0+M_t+\int_0^t \alpha_u^\F \ud \langle M\rangle_u, \quad t \in [0,T],
\]
where
\[
M_t= \int_0^t S_u \sigma_1(u, S_u) \ud W^1_u \quad \mbox{and}\quad  \alpha_t^\F= \frac{\mu_1(t, X_{t}, S_{t})}{S_{t} \sigma_1^2(t, S_{t})}, \quad t \in [0,T].
\]
According to~\cite[Lemma 2.2]{kxy2006}, $S$ also satisfies the structure condition with respect to the filtration $\bH$, which is given by
\[
S_t = S_0 + N_t + \int_0^t {}^p\alpha_u^\F \ud \langle N \rangle_u, \quad t \in [0,T],
\]
where, as usual, ${}^pY$ denotes the $\bH$-predictable projection of a given (integrable) process $Y$ under $\P$, and $N$ is the $(\bH, \P)$-martingale that
satisfies
\[
\ud N_t = S_t \sigma_1(t, S_t) \left(\ud W^1_t + \frac{\mu_1(t, X_t, S_t)-{}^p\mu_1(t, X_t, S_t)}{\sigma_1(t, S_t)}\ud t\right).
\]
Finally, we define the process $I=\{I_t,\ t \in [0,T]\}$ by setting
\begin{equation}\label{eq:innovation}
I_t:= W^1_t+\int_0^t \frac{\mu_1(u, X_u, S_u)-{}^p\mu_1(u, X_u, S_u)}{\sigma_1(u, S_u)}\ud u
\end{equation}
for each $t \in
[0,T]$. It is known that $I$ is an $(\bH, \P)$-Brownian motion called the {\em innovation process} (see e.g.~\cite{FKK} and~\cite{K}).  Then, $S$
satisfies the SDE
\begin{equation*}
\ud S_t= S_t \left( {}^p\mu_1(t, X_t, S_t) \ud t + \sigma_1(t, S_t) \ud I_t \right),\quad S_0=s>0.
\end{equation*}

\subsubsection{The $\bH$-pseudo optimal strategy}

Notice that, thanks to \eqref{ass:boundedness_1}, the Novikov condition
\begin{equation*} \label{eq:novikov}
\esp{\exp\left\{\frac{1}{2}\int_0^T \frac{\mu^2_1(t, X_t, S_t)}{\sigma^2_1(t, S_t)}\ud t\right\}}<\infty
\end{equation*}
is satisfied.

Therefore, we can introduce the MMM $\P^*$ for the underlying model whose density is given by
\[
\left.\frac{\ud \P^*}{\ud \P}\right|_{\F_T}=L_T,
\]
where the process $L=\{L_t,\ t \in [0,T]\}$ is defined by $\ds L_t=\E\left(-\int \frac{\mu_1(u, X_u, S_u)}{\sigma_1(u, S_u)} \ud W^1_u \right)_t$ for every $t \in
[0,T]$.

As pointed out at the beginning of Section \ref{sec:markovian_models}, we assume condition \eqref{hp:filtrazione_esempi} to compute the $\bH$-pseudo optimal strategy for the contingent claim $\xi=H(T,S_T)$. Note that, under \eqref{hp:filtrazione_esempi} the $(\bH, \P^*)$-optional projection of a process $D$ can be written as $\pi(D)$. \\
By the Girsanov Theorem, the process $\widetilde{W}=\{\widetilde{W}_t, \ t \in [0,T]\}$, defined by
\begin{equation*}\label{eq:W_tilde}
\widetilde{W}_t:=W^1_t+ \int_0^t
\frac{\mu_1(u, X_u, S_u)}{\sigma_1(u, S_u)} \ud u, \quad t \in [0,T],
\end{equation*}
is an $(\bF, \P^*)$-Brownian motion. On the other hand,  $\ds \widetilde{W}_t=I_t+ \int_0^t
\frac{{}^p\mu_1(u,X_u,S_u)}{\sigma_1(u, S_u)} \ud u$ for every $t \in [0,T]$, which in turn implies that $\widetilde{W}$ is an $(\bH, \P^*)$-Brownian motion, since
all the processes involved are $\bH$-adapted.
Then, under the MMM $\P^*$, the system \eqref{eq:sistema1} can be written as
\begin{equation}\label{eq:sistema1b}
\left\{
\begin{aligned}
\ud X_t&= \mu_0(t, X_t) \ud t + \sigma_0(t,X_t) \ud W^0_t , \quad X_0=x\in \R, \\
\ud S_t&=S_{t}\sigma_1(t, S_t)\ud \widetilde{W}_t,\quad S_0=s>0,
\end{aligned}
\right.
\end{equation}
where $W^0$ and $\widetilde{W}$ turn out to be correlated $(\bF, \P^*)$-Brownian motions with correlation coefficient $\rho$.
It is important to stress that, since the change of probability measure is Markovian, the pair $(X,S)$ is also an $(\bF, \P^*)$-Markov process
(see~\cite[Proposition 3.4]{cg09}).
The following result provides the $(\bF, \P^*)$-generator of the Markovian pair $(X,S)$.
\begin{proposition}\label{lemma-generatore-Lcont}
Assume that
\begin{equation}\label{hp:generatore_cont}
\bE^{\P^*}\left[\int_0^T\left\{|\mu_0(t,X_t)|+ \sigma_0^2(t,X_t)\right\}\ud t\right]<\infty.
\end{equation}
Then, the pair $(X,S)$ is an $(\bF, \P^*)$-Markov process with generator
\begin{equation}
\begin{split}
\L^1_{X,S} f(t,x,s) =& \frac{\partial f}{\partial t}+\mu_0(t,x)\frac{\partial f}{\partial x}+ \frac{1}{2} \sigma_0^2(t,x) \frac{\partial^2 f}{\partial
x^2}+ \rho \sigma_0(t,x) \sigma_1(t,s) s  \frac{\partial^2 f}{\partial x\partial s}
  +\frac{1}{ 2} \sigma_1^2(t,s)\, s^2  \frac{\partial^2 f}{\partial s^2} \label{generatore-L1}
  \end{split}
\end{equation}
for every function $f \in  \C^{1,2,2}_b([0,T]\times \R \times \R^+)$.
Moreover, the following decomposition holds
\[
f(t,X_t,S_t)=f(0,X_0,S_0)+\int_0^t\L^1_{X,S} f(r,X_r,S_r) \ud r + M^{1,f}_t
\]
where $M^{1,f}=\{M_t^{1,f},\ t \in [0,T]\}$ is the $(\bF, \P^*)$-martingale given by
 \begin{gather}
  \ud M^{1, f}_t = \frac{\partial f}{\partial x} \sigma_0(t,X_t)\ud W^{0}_t + \frac{\partial f}{\partial s} \sigma_1(t,S_t) S_t\ud \widetilde{W}_t.
  \label{eq:M1f}
   \end{gather}
\end{proposition}
The proof is postponed to Appendix \ref{appendix:b}.

We recall that in the continuous case the Galtchouk-Kunita-Watanabe decomposition \eqref{GKW2} of $\xi$ under the MMM $\P^*$  and the F\"{o}llmer-Schweizer decomposition \eqref{eq:fsdecomposition} coincide and therefore the process $V^\H$ provides the optimal value process $V(\psi^*)$ . Then, to compute $\beta^\H$ we will apply \eqref{eq:optimal_strategy_cont} and \eqref{eq:VH}, which requires the knowledge of the filter. For the partially observable system \eqref{eq:sistema1b}, the filter dynamics is described by the
Kushner-Stratonovich equation given by \eqref{eq:ks_continuo} in Appendix \ref{appendix:a}.
Then, under assumptions \eqref{ass:boundedness_1} and  \eqref{hp:generatore_cont} 
for each $t \in [0,T]$ we get
\begin{equation} \label{basta}
\beta^\H_t=\frac{\ud {}^\bH \langle \pi(g), S \rangle_t}{\ud  {}^\bH \langle S \rangle_t}=\frac{h_{t^-} (g)}{S_{t^{-}}
 \sigma_1(t,S_{t^{-}})}= \frac{\rho  \pi_{t^{-}} \left(\sigma_0 \frac{\partial g }{\partial x}\right) + S_{t^{-}}
 \sigma_1(t,S_{t^{-}})\pi_{t^{-}}\left(\frac{\partial g }{\partial s}\right)}{S_{t^{-}} \sigma_1(t,S_{t^{-}})},
\end{equation}
where $h_{t}(g)$ is defined in  \eqref{eq:h_continuo} with the choice $f=g$ and $g(t,x,s)$ is the solution of the problem \eqref{eq:problema}, with
$\L^*_{X,S}=\L^1_{X,S}$ being the operator given in  \eqref{generatore-L1}.

Now, we check that equation \eqref{eq:formula_strategia_caso_continuo} holds; in other terms, that $\beta^\H$ coincides with the
$(\bH, \P^*)$-predictable projection  of  $\beta^\F$, which represents the first component of the pseudo-optimal strategy
$\psi^\F=(\beta^\F,\eta^\F)$ under complete information.\\
To derive an expression for $\beta^\F$ we consider the process $V^\F=\{V_t^\F,\ t \in
[0,T]\}$ given by
$$
V_t^\F=\bE^{\P^*}\left[H(T,S_T)|\F_t\right], \qquad \forall t \in [0,T]
$$
thanks to Corollary \ref{cor:strategia_ottima} with the choice $\bH =\bF$; consequently,
\[
\beta^\F_t=\frac{\ud {}^\bF \langle V^\F, S\rangle_t }{\ud {}^\bF \langle S\rangle_t}
 \qquad \forall t \in [0,T].
 \]
We observe that $ V^\F$ coincides with the process $\{g(t,X_t,S_t),\ t \in [0,T]\}$, then by It\^{o}'s formula we get
\[
V_t^\F=g(t,X_t,S_t)=\int_0^t \left \{ \frac{\partial g}{\partial x} \sigma_0(u,X_u)\ud W^{0}_u + \frac{\partial g}{\partial s}
\sigma_1(u,S_u) S_u\ud \widetilde{W}_u \right\}, \quad t \in [0,T],
\]
and computing explicitly the sharp brackets ${}^\bF \langle V^\F, S\rangle$ and ${}^\bF \langle S\rangle$, we obtain
\[
\beta_t^\F=\frac{\rho  \sigma_0(t,X_{t^{-}})  \frac{\partial g }{\partial x} + S_{t^{-}}\sigma_1(t,S_{t^{-}}) \frac{\partial g }{\partial s}}{S_{t^{-}}
\sigma_1(t,S_{t^{-}})},    \quad t \in [0,T].
\]
Finally, taking \eqref{basta} and the definition of the filter into account, we get that
$\beta^\H = {}^{p,*}\beta^\F$, where  ${}^{p,*}\beta^\F$ is the $(\bH, \P^*)$-predictable projection of the process $\beta^\F$.

\subsection{A pure jump market model}\label{sec:jump_model}

We now consider the case where the risky asset price dynamics is described by a pure jump process that depends on some unobservable process $X$, given by a Markovian jump-diffusion having common jump times with $S$.
More precisely,
\begin{equation}\label{eq:sistema2}
\left\{
\begin{aligned}
\ud X_t&= \mu_0(t, X_t) \ud t + \sigma_0(t,X_t) \ud W^0_t + \int_Z K_0(\zeta; t, X_{t^-}) \N(\ud t, \ud \zeta), \quad X_0=x\in \R \\
\ud S_t&=S_{t^-}\int_Z K_1(\zeta;t, X_{t^-}, S_{t^-}) \N(\ud t, \ud \zeta),\quad S_0=s>0.
\end{aligned}
\right.
\end{equation}
Here $\N(\ud t, \ud \zeta)$, $(t,\zeta)\in [0,T]\times Z$, with $Z \subseteq \R$, is an $(\bF, \P)$-Poisson random measure having nonnegative
intensity $\eta(\ud \zeta)\ud t$.
The measure $\eta(\ud \zeta)$, defined on the measurable space $(Z, \mathcal Z)$, is  $\sigma$-finite. The corresponding  $(\bF, \P)$-compensated random measure is given by
\begin{equation*} \label{def:cm}
\widetilde \N(\ud t, \ud \zeta)=\N(\ud t, \ud\zeta)-\eta(\ud \zeta)\ud t.
\end{equation*}
The process $W^0$ is an $(\bF,\P)$-Brownian motion independent of $\N(\ud t, \ud \zeta)$ and  $\mu_0(t,x)$, $\sigma_0(t,x) >0$,  $K_0(\zeta;t,x)$ and $K_1(\zeta; t,x,s)$ are $\R$-valued
measurable functions of their arguments such that a unique strong solution for the system \eqref{eq:sistema2} exists.
In particular, this implies that
the pair $(X, S)$ is an $(\bF, \P)$-Markov process.

Note that if the set $
\{\zeta\in Z : K_1(\zeta;t,X_{t^-},S_{t^-})\neq 0 \ \mbox{ and } \ K_0(\zeta;t,X_{t^-})\neq 0\}$
is not empty, $S$ and $X$ have common jump times. This feature may describe, for example, catastrophic events that affect at the same time the
stock price and the hidden state variable that influences it.

We assume that
\begin{equation}\label{ass:boundedness_1b}
 K_1(\zeta;t,X_t, S_t)<c_4, \ \forall (t, \zeta) \in [0,T]\times Z,
 \end{equation}
  for some constant $c_4$, and, to ensure nonnegativity of $S$ we also assume that $K_1(\zeta;t, X_t, S_t) + 1 > 0$ $\P$-a.s..

To describe the jumps of $S$, we introduce the integer-valued random measure

\begin{equation*}\label{m}
m(\ud t, \ud z) =  \sum_{r: \Delta S_r \neq 0}
  \delta _{\{r, \Delta S_r\}} (\ud t, \ud z),
\end{equation*}
where $\delta_a$ denotes as usual the Dirac measure at  point $a$.  Note that the following equality holds
\begin{equation*}\label{misura_m}
\int_0^t\int_{\mathbb{R}}z \ m(\ud u,\ud z)=\int_0^tS_{u^-}\int_{Z}K_1(\zeta;u, X_{u^-}, S_{u^-})\N(\ud u,\ud \zeta)
\end{equation*}
and, in general, for any measurable function $\gamma:\mathbb{R}\rightarrow \mathbb{R}$, we get that
\begin{equation*}\label{integrale_rispetto_m}
\int_0^t\int_{\mathbb{R}}\gamma(z)m(\ud s,\ud z)=\int_0^t\int_{Z}\I_{D_u}(\zeta)  \gamma \left(S_{u^-}K_1(\zeta;u, X_{u^-}, S_{u^-})\right)\N(\ud u,\ud \zeta),
\end{equation*}
where $D_t:=\{\zeta \in Z: K_1(\zeta;t, X_{t^-}, S_{t^-})\neq 0\}$. From now on we assume that
\begin{equation} \label{inten}
\esp{\int_0^T  \eta(D_t)\ud t} < \infty, \mbox{ and } \eta(D_t)>0, \quad \forall t \in [0,T].
\end{equation}
\begin{remark}
Recall that $\nu^\bF(\ud t, \ud z)$ denotes the $(\bF, \P)$-predictable dual projection of the random measure $m(\ud t, \ud z)$. Under condition
\eqref{inten}, it is proved in~\cite{cg06} and~\cite{c06} that $\nu^\bF(\ud t, \ud z)$, is absolutely continuous with respect to  the Lebesgue measure,
that is, $\nu^\bF(\ud t, \ud z)=\nu^\bF_t(\ud z)\ud t$
where, for any $\A \in\mathcal{B}(\mathbb{R})$, $\nu^\bF_t(\A) = \eta(D^\A_t)$ with $D^\A_t:=\{\zeta \in Z: K_1(\zeta;t, X_{t^-}, S_{t^-}) \in \A \setminus \{0\} \}$.\\
In particular, $\nu^\bF_t(\mathbb{R}) = \eta(D_t)$, where $D_t=D^{\R}_t$, for every $t \in [0,T]$,  provides the $(\bF, \P)$-intensity of the point
process $m((0,t]\times \R)$ which counts the total number of jumps of $S$ up to time $t$.\end{remark}

\subsubsection{Structure conditions of the stock price $S$ with respect to $\bF$ and $\bH$.}
We observe that the semimartingale $S$ admits the following canonical $\bF$-decomposition
\[
S_t=S_0+M_t+\Gamma_t, \quad t \in [0,T],
\]
where $M$ is the square-integrable $(\bF, \P)$-martingale given by
\[
\ud M_t= S_{t^-} \int_Z K_1(\zeta; t, X_{t^-}, S_{t^-} )\widetilde \N(\ud t, \ud \zeta)=\int_{\R}z (m(\ud t, \ud z)-\nu^\bF_t(\ud z)\ud t)
\]
and $\Gamma=\{\Gamma_t,\ t \in [0,T]\}$ is an $\R$-valued nondecreasing $\bF$-predictable finite variation process  satisfying
\[
\ud \Gamma_t= S_{t^-} \int_Z K_1(\zeta;t, X_{t}, S_{t})\eta(\ud \zeta) \ud t=\int_{\R}z \ \nu^\bF_t(\ud z)\ud t.
\]
The $\bF$-predictable quadratic variation of $M$ is absolutely continuous with respect to the Lebesgue measure; in fact, $\ud \langle M\rangle_t= a_t
\ud t$, where
$a_t= S^2_{t^-} \int_Z K_1^2(\zeta;t, X_{t^-}, S_{t^-})\eta(\ud \zeta)=\int_{\R}z^2\nu^\bF_t(\ud z)\ud t$, for every $t \in [0,T]$.

In the sequel, we will assume that the $(\bF, \P)$-intensity of the point process $m((0,t]\times \R)$, which counts the jumps of $S$ up to time $t$,  is strictly
positive, i.e. $\nu^\bF_t(\mathbb{R}) = \eta(D_t)>0 $ $\P$-a.s. for every $t \in [0,T]$.
Then, the $(\bF, \P)$-semimartingale $S$ satisfies  the structure condition with respect to $\bF$, i.e.
\begin{gather*}
S_t=S_0+M_t+\int_0^t\alpha^\F_s \ud \langle M\rangle_s, \quad t \in [0,T],
\end{gather*}
where
\begin{gather*}
\alpha^\F_t= \frac{ \int_{Z}K_1(\zeta; t, X_{t}, S_{t})\eta(\ud \zeta)}{S_{t^-} \int_{Z}K_1^2(\zeta; t, X_{t}, S_{t})\eta(\ud \zeta)}=\frac{\int_{\R}z \nu^\bF_t(\ud
z)}{\int_{\R}z^2 \nu^\bF_t(\ud z) }, \quad t \in [0,T].
\end{gather*}
Notice that since \eqref{inten} holds true, then $\alpha^\F$ is well defined and $\esp{\int_0^T(\alpha_t^\F)^2 \ud \langle M \rangle_t}<\infty$.

Moreover, $S$ also admits the canonical $\bH$-decomposition, which is given by
\[
S_t=S_0+\int_0^t z \left[m (\ud u, \ud z)-\nu_u^\bH(\ud z)\ud u\right] + \int_0^t z \nu_u^\bH(\ud z)\ud u, \quad t \in [0,T],
\]
where $\nu_t^\bH(\ud t,\ud z)=\nu_t^\bH(\ud z) \ud t$ denotes the $(\bH, \P)$-predictable dual projection of $m(\ud t,\ud z)$,
and satisfies the structure condition with respect to $\bH$, i.e.
\[
S_t=S_0+N_t+\int_0^t \alpha^\H_s\ud \langle N\rangle_s, \quad t \in [0,T],
\]
with
\begin{equation*}
N_t =\int_0^t \int_\R z \left[m (\ud u, \ud z)-\nu_u^\bH(\ud z)\ud u\right],
\quad \alpha^\H_t=\frac{\int_\R z \nu_t^\bH(\ud z)}{ \int_{\R} z^2 \nu_t^\bH(\ud z)}, \quad t \in [0,T].
\end{equation*}

\subsubsection{The $\bH$-pseudo optimal strategy}

To introduce the MMM $\P^*$ for the underlying pure jump market model, we also assume that if $t$ is a jump time of $S$, then
\begin{equation*}
\alpha^\F_t \Delta M_t = K_1(\zeta; t,X_{t^-},S_{t^-})\frac{\int_{Z}K_1(\zeta; t, X_{t^-},S_{t^-})\eta(\ud
\zeta)}{\int_{Z}K_1^2(\zeta;t,X_{t^-},S_{t^-})\eta(\ud \zeta)}<1
\end{equation*}
and
\begin{equation}\label{ps2}
\bE\left[ \exp\left\{\int_0^T (\alpha^\F_t)^2 \ud \langle M \rangle_t \right\}\right] =
\bE\left[ \exp\left\{\int_0^T \frac{\left(\int_{Z}K_1(\zeta; t, X_{t^-},S_{t^-}) \eta(\ud \zeta)\right)^2}{ \int_{Z}K_1^2(\zeta; t, X_{t^-},S_{t^-})
\eta(\ud \zeta)}\ud t \right\}\right]<\infty.
\end{equation}

\begin{remark}
It is worth stressing that a sufficient condition for \eqref{ps2} is that $\ds \bE\left[ \exp\left\{\int_0^T \eta(D_t)\ud t \right\}\right]<\infty$. Indeed,
\begin{align*}
\bE\left[ \exp\left\{\int_0^T \frac{\left(\int_{Z}K_1(\zeta; t, X_{t^-},S_{t^-}) \eta(\ud \zeta)\right)^2}{ \int_{Z}K_1^2(\zeta; t, X_{t^-},S_{t^-})
\eta(\ud \zeta)}\ud t \right\}\right]&\leq \bE\left[ \exp\left\{\int_0^T \frac{\eta(D_t) \int_{Z}K_1^2(\zeta; t, X_{t^-},S_{t^-}) \eta(\ud \zeta)}{ \int_{Z}K_1^2(\zeta; t, X_{t^-},S_{t^-})
\eta(\ud \zeta)}\ud t \right\}\right]\\
&=\bE\left[ \exp\left\{\int_0^T \eta(D_t)\ud t \right\}\right].
\end{align*}
\end{remark}

Hence, we can apply the Ansel-Stricker Theorem and define the change of probability measure $\ds \left.\frac{\ud \P^*}{\ud \P}\right|_{\F_T}=L_T$,
where the process $L$ is given by
$$
L_t= \doleans{-\int \alpha^\F_r \ud M_r}_t=\doleans{-\int \int_{Z}\alpha^\F_r S_{r^-} K_1(\zeta;r,X_{r^-},S_{r^-}) \widetilde{\N}(\ud r, \ud
\zeta)}_t,
$$
for each $t \in [0,T]$.
Under the MMM $\P^*$, the dynamics of the pair $(X,S)$ becomes
\begin{gather}\label{eq:sistema2b}
\left\{
\begin{split}
\ud X_t&= \mu_0(t, X_t) \ud t + \sigma_0(t,X_t) \ud W^0_t + \int_Z K_0(\zeta; t, X_{t^-}) \N(\ud t,\ud \zeta), \quad X_0=x\in \R\\
\ud S_t &= S_{t^-}\int_Z K_1(\zeta; t, X_{t^-},S_{t^-}) \widetilde{\N}^*(\ud t,\ud \zeta), \quad S_0=s>0,
\end{split}
\right.
\end{gather}
where
\begin{equation*} \label{tilde}
\widetilde{\N}^*(\ud t,\ud \zeta)=\N(\ud t,\ud \zeta)- \eta^{*}_t(\ud \zeta)\ud t
\end{equation*}
and
\[
\eta_t^*(\ud \zeta)\ud t=\left[1-\alpha^\F_t S_{t^-} K_1(\zeta;t,X_{t^-},S_{t^-})\right]\eta(\ud \zeta)\ud t
\]
is the $(\bF,\P^*)$-predictable dual projection of the random measure $\N(\ud t, \ud \zeta)$.
In the sequel we will assume the following integrability condition holds:
\begin{gather}
\espp{\int_0^T  \left(|\mu_0(t, X_t)|+\sigma_0^2(t, X_t)+\int_Z |K_0(\zeta;t, X_{t^-})|\eta_t^*(\ud
\zeta)\right)\ud t}<  \infty. \label{integrab}
\end{gather}
Since the change of probability measure is Markovian, the pair $(X,S)$ is still an $(\bF,\P^*)$-Markov process (see~\cite[Proposition 3.4]{cg09}), whose generator is derived in the following proposition.
\begin{proposition}\label{lemma-generatore-L2}
Assume \eqref{integrab} and
\begin{equation}\label{hp:L2}
\mathbb{E}^{\P^*}\left[\int_0^T\left\{\eta_t^*(D^0_t) + \eta_t^*(D_t) \right\}\ud t\right]
<\infty,
\end{equation}
where $D^0_t=\{\zeta\in Z: K_0(\zeta; t, X_{t^-})\neq 0\}$ and $D_t:=\{\zeta \in Z: K_1(\zeta; t, X_{t^-}, S_{t^-})\neq 0\}$. Then, the pair $(X,S)$ is
an $(\bF, \P^*)$-Markov process with generator
\begin{equation}
\L^2_{X,S} f(t,x,s) = \frac{\partial f}{\partial t}+\mu_0(t,x)\frac{\partial f}{\partial x}+ \frac{1}{2} \sigma_0^2(t,x) \frac{\partial^2 f}{\partial
x^2} +  \int_Z \Delta f(\zeta;t,x,s)\eta_t^*(\ud \zeta)-\frac{\partial f}{\partial s} s \int_{Z}K_1(\zeta;t,x,s)\eta_t^*(\ud
\zeta),\label{generatore-L2}
\end{equation}
where
\[
\Delta f (\zeta;t,x,s):= f \big (t,x+K_0(\zeta;t,x), s( 1 +K_1(\zeta; t,x, s)) \big)-f(t,x,s).
\]
Moreover, the following semimartingale decomposition holds
\[
f(t,X_t,S_t)=f(0,X_0,S_0)+\int_0^t\L^2_{X,S} f(r,X_r,S_r) \ud r + M^{2,f}_t,
\]
where $M^{2,f}=\{M_t^{2,f},\ t \in [0,T]\}$ is the $(\bF, \P^*)$-martingale given by
 \begin{equation}\label{eq:Mf2}
\ud M^{2,f}_t = \frac{\partial f}{\partial x} \sigma_0(t,X_t) \ud W^{0}_t  + \int_Z \Delta f(\zeta;t, X_{t^-},S_{t^-})\widetilde{ \N}^*(\ud t, \ud
\zeta).
 \end{equation}

\end{proposition}
The proof is postponed to Appendix \ref{appendix:b}.

As before, we assume \eqref{hp:filtrazione_esempi}, so that we can compute the $\bH$-pseudo optimal strategy via \eqref{eq:strategia_filtro}.
Taking the  dynamics of the filter for the pure jump model  given by \eqref{eq:ks_jump} in Appendix \ref{appendix:a} into account, under the
assumptions \eqref{integrab} and \eqref{hp:L2} we get
\begin{equation*}
\beta^\H_t=\widetilde \beta^\H_t+\phi^\H_t , \quad t \in [0,T],
\end{equation*}
where
\begin{align}
\widetilde \beta^\H_t&=\frac{\ud  \langle \pi(g), S \rangle_t^{*,\bH}}{\ud  \langle S \rangle_t^{*,\bH}}=\frac{\int_\R z \ w^g(t, z) \  \nu_t^{\bH,*}(\ud z)}{\int_\R
z^2 \ \nu_t^{\bH,*}(\ud z)},  \ t \in [0,T], \nonumber\\
 \phi^\H_t&=  \frac{\ud {}^\bH\langle  \sum_{r \leq \cdot} \Delta G_r \Delta S_r,  -\int_0^\cdot \alpha^\H_r \ud N_r \rangle_t}
{\ud {}^{\bH} \langle S \rangle_t }=\frac{\alpha_t^\H \int_\R z^2 \left\{ \widetilde\beta^\H_t z -w^g(t,z)  \right\} \nu^{\bH}_t (\ud z)}{\int_\R z^2 \nu^\bH_t(\ud z)}, \ t \in [0,T]. \label{eq:phiH}
\end{align}
Here $g(t,x,s)$ is the solution to \eqref{eq:problema} with $\L^*_{X,S}=\L^2_{X,S}$ and
$w^g(t,z)$ is given by
$$
w^g(t,z)= \frac{\ud \pi_{t^-} (g \nu^{\bF,*})}{\ud \nu^{\bH,*}_t} (z) - \pi_{t^-}(g) +
  \frac{ \ud \pi_{t^-} (\overline{\L} g)}{\ud \nu^{\bH,*}_{t}} (z),\quad t \in [0,T].
 $$
Moreover, $ \nu^{\bF,*}_t(\ud z) \ud t= (1-\alpha^\F_t z) \nu_t^\bF(\ud z) \ud t$  and $\nu^{\bH,*}_t(\ud z) = (1- \alpha^\H_t z) \nu^\bH_t(\ud z) \ud
t$ denote  respectively the $(\bF,\P^*)$-predictable and the $(\bH,\P^*)$-predictable dual projections
of the random measure $m(\ud t, \ud z)$,  and by~\cite[Proposition 2.2]{c06}, $\nu^{\bH,*}_t(\ud z)= \pi_{t^-}(\nu^{\bF,*}(\ud z))$.
Let us observe that the operator  $\overline{\L}$, defined in Proposition \ref{prop:ks_jumpdiff},
 takes common jump times between $S$ and $X$ into account.

It is worth observing that the $(\bH,\P)$-predictable dual projection $\nu^\bH_t(\ud z)\ud t$ of the measure $m(\ud t,\ud z)$ appearing in \eqref{eq:phiH},  can be written in terms
of the filter under the real-world probability measure $\P$. Indeed, set $\widetilde \pi_t(f):=\esp{f(t,X_t,S_t)|\H_t}$, for every $t \in [0,T]$. Then,
$\nu^\bH_t(\ud z)= \widetilde \pi_{t^-} ( \nu^\bF(\ud z))$(see
again~\cite[Proposition 2.2]{c06} for the proof). Therefore, in presence of jumps we also need the knowledge the filter dynamics under $\P$.
The Kushner-Stratonovich equation satisfied by $\widetilde \pi$ is given by \eqref{eq:ks_P} in Appendix
\ref{appendix:a}.

\subsection{A jump-diffusion market model}\label{sec:jumpdiff_model}
In the last part of this overview on Markovian models we wish to discuss the case of a jump-diffusion market model, where the risky asset  price
dynamics is described by a geometric jump diffusion, where as usual, $X$ represents an unobservable stochastic factor that influences the dynamics of $S$, and it is modeled by a Markovian jump-diffusion having common jump times with S. Precisely, we consider the following system of SDEs:
\begin{equation}\label{eq:sistema3}
\left\{
\begin{aligned}
\ud X_t&= \mu_0(t, X_t) \ud t + \sigma_0(t,X_t) \ud W^0_t + \int_Z K_0(\zeta; t, X_{t^-}) \N(\ud t, \ud \zeta), \quad X_0=x\in \R \\
\ud S_t&=S_{t^-}\left(\mu_1(t, X_t, S_t)\ud t+\sigma_1(t,S_t)\ud W^1_t+\int_Z K_1(\zeta;t, X_{t^-}, S_{t^-}) \N(\ud t, \ud \zeta)\right),\quad
S_0=s>0,
\end{aligned}
\right.
\end{equation}
where $\N(\ud t, \ud \zeta)$ is an  $(\bF, \P)$-Poisson random measure with mean measure $\eta_t(\ud \zeta)\ud t$ according to the previous models, $W^0$ and $W^1$ are $(\bF,\P)$-Brownian motions independent of $\N(\ud t, \ud \zeta)$ such that $\langle W^0,W^1\rangle_t=\rho t$, for every $t \in [0,T]$,
with $\rho \in [-1,1]$, the coefficients $\mu_0(t,x), \mu_1(t,x,s)$, $\sigma_0(t,x)>0, \sigma_1(t,x,s)>0$, $K_0(\zeta;t,x)$ and $K_1(\zeta; t,x,s)$ are
$\R$-valued measurable functions of their arguments such that a unique strong solution for the system \eqref{eq:sistema3} exists, see for
instance~\cite{OS}. In particular, this implies that the pair $(X, S)$ is an $(\bF, \P)$-Markov process.

To ensure nonnegativity of $S$ we assume that $K_1(\zeta;t, X_t, S_t) + 1 > 0$ $\P$-a.s. for every $(t,\zeta)\in [0,T] \times Z$.
For simplicity we also take
\begin{equation}\label{ass:boundedness_2}
\mu_1(t, X_t, S_t)<c_1, \ 0<c_2<\sigma_1(t, S_t)<c_3 \ \mbox{ and } \ K_1(\zeta;t, X_t, S_t)<c_4, \quad  \forall t \in [0,T], \ \zeta \in Z,
\end{equation}
for some constants $c_1, c_2, c_3, c_4$.

Finally, according to the previous model, we assume that
\begin{equation}\label{int_2}
\esp{\int_0^T \eta(D_t) \ud t}<\infty.
\end{equation}

\subsubsection{Structure conditions of the stock price $S$ with respect to $\bF$ and $\bH$.}

The canonical $\bF$-decomposition of $S$ with respect to $\bF$ is
given by
\[
S_t=S_0+M_t+\Gamma_t, \quad t \in [0,T],
\]
where $M$ is the square-integrable $(\bF, \P)$-martingale
\[
\ud M_t= S_t \sigma_1(t, S_t) \ud W^1_t + S_{t^-} \int_Z K_1(\zeta; t, X_{t^-}, S_{t^-} ) \widetilde \N(\ud t, \ud \zeta)=S_t \sigma_1(t, S_t) \ud W^1_t +\int_\R z (m(\ud t, \ud
z)-\nu^\bF_t(\ud z)\ud t)
\]
and $\Gamma$ is the following $\R$-valued nondecreasing $\bF$-predictable finite variation process
\[
\ud \Gamma_t= S_{t^-}\left\{\mu_1(t, X_{t^-}, S_{t^-} )+\int_Z K_1(\zeta;t, X_{t^-}, S_{t^-})\eta(\ud \zeta)\right\} \ud t = \left\{S_{t^-}  \mu_1(t, X_{t^-}, S_{t^-})+\int_\R z \nu^\bF_t(\ud z)\right\} \ud t.
\]
We note that the $\bF$-predictable quadratic variation of $M$ is absolutely continuous with respect to the Lebesgue measure, that is,
$\ud \langle M\rangle_t= a_t \ud t$ with
\[
a_t= S^2_{t^-} \left(\sigma_1^2(t, S_{t^-}) + \int_Z K_1^2(\zeta;t,X_{t^-}, S_{t^-})\eta(\ud \zeta)\right) =S^2_{t^-} \sigma_1^2(t, S_{t^-}) +\int_\R z^2 \nu^\bF_t(\ud z), \quad t \in [0,T].
\]
Then, the semimartingale $S$ satisfies the structure condition with respect to $\bF$ and with respect to $\bH$, which are respectively given by
\[
S_t=S_0+M_t+\int_0^t\alpha^\F_s \ud \langle M\rangle_s, \quad t \in [0,T]
\]
and
\[
S_t=S_0+N_t+\int_0^t \alpha^\H_s \ud \langle N\rangle_s, \quad t \in [0,T],
\]
where
\begin{gather}\label{eq:alphaF}
\alpha^\F_t= \frac{\mu_1(t, X_{t^-}, S_{t^-})+ \int_{Z}K_1(\zeta; t, X_{t^-}, S_{t^-}) \eta(\ud \zeta)}{S_{t^-}  \left(\sigma_1^2(t, S_{t^-})+\int_{Z}K_1^2(\zeta; t, X_{t^-}, S_{t^-}) \eta(\ud
\zeta)\right)}=\frac{S_{t^-} \mu_1(t, X_{t^-}, S_{t^-})+\int_\R z \ \nu^\bF_t(\ud z)}{S^2_{t^-} \sigma^2_1(t, S_{t^-}) + \int_\R z^2 \nu^\bF_t(\ud z)}, \quad t \in
[0,T],
\end{gather}
and
\begin{gather*}
\ud N_t= S_{t} \sigma_1(t, S_t)\ud I_t + \int_\R z (m(\ud t, \ud z)- \nu_t^\bH(\ud z)\ud t), \quad \alpha^\H_t=\frac{S_{t^-}{}^p \mu_1(t, X_{t^-}, S_{t^-}) + \int_{\R}z \nu_t^\bH(\ud z)
}{S^2_{t^-} \sigma_1^2(t, S_{t^-})+\int_{\R}z^2 \nu_t^\bH(\ud z)},
\end{gather*}
for every $t \in [0,T]$, where $I$ is the $(\bH, \P^*)$-Brownian motion defined in \eqref{eq:innovation}. Notice that, under the assumptions on the coefficients of the dynamics of $S$, $\alpha^\F$ is well defined and because of \eqref{int_2} also $\esp{\int_0^T(\alpha_t^\F)^2 \ud \langle M \rangle_t}<\infty$ is fulfilled.

\subsubsection{The $\bH$-pseudo optimal strategy}

To introduce the MMM $\P^*$ for the underlying market model, we assume that at every jump time of $S$, the following condition
\begin{equation*}\label{jum}
\alpha^\F_t \Delta M_t  = K_1(\zeta; t,X_{t^-},S_{t^-})\frac{\mu_1(t,X_{t^-}, S_{t^-})+ \int_{Z}K_1(\zeta; t,X_{t^-},S_{t^-})\eta(\ud
\zeta)}{\sigma_1^2(t,S_{t^-})+\int_{Z}K_1^2(\zeta;t,X_{t^-},S_{t^-})\eta(\ud \zeta)}<1
\end{equation*}
holds and that
\begin{equation}\label{ps3}
\bE\left[\exp\left\{\frac{1}{2}\int_0^T (\alpha^\F_t)^2 \ud \langle M^c \rangle_t + \int_0^T (\alpha^\F_t)^2 \ud \langle M^d \rangle_t\right\}\right] <
\infty.
\end{equation}

\begin{remark}
A sufficient condition for \eqref{ps3} is given by $\ds \esp{\exp\left\{2 \int_0^T \eta(D_t)\ud t\right\}}<\infty$. Indeed,
\begin{align*}
&\bE\left[\exp\left\{\frac{1}{2}\int_0^T (\alpha^\F_t)^2 \ud \langle M^c \rangle_t + \int_0^T (\alpha^\F_t)^2 \ud \langle M^d \rangle_t\right\}\right]\\
&\qquad \leq \esp{\exp\left\{\int_0^T\frac{\left(\mu_1(t, X_{t^-}, S_{t^-})+ \int_{Z}K_1(\zeta; t, X_{t^-}, S_{t^-}) \eta(\ud \zeta)\right)^2}{\sigma_1^2(t, S_{t^-})+\int_{Z}K_1^2(\zeta; t, X_{t^-}, S_{t^-}) \eta(\ud
\zeta)} \ud t\right\}}\\
&\qquad \leq \esp{\exp\left\{2\int_0^T\left(\frac{\mu_1^2(t, X_{t^-}, S_{t^-})}{\sigma_1^2(t, S_{t^-})}+\eta (D_t)\right)  \ud t\right\}}\\
& \qquad \leq \exp\left\{2 \ \frac{c_1^2}{c_2^2} \ T\right\} \esp{\exp\left\{2\int_0^T\eta (D_t)  \ud t\right\}}
\end{align*}

\end{remark}
Then, we can apply the Ansel-Stricker Theorem and define the change of probability measure $\ds \left.\frac{\ud \P^*}{\ud \P}\right|_{\F_T}=L_T$ where
the process $L$ is given by $\ds L_t= \doleans{- \int \alpha^\F_r \ud M_r}_t$ with $ t \in [0,T]$.
Under the MMM $\P^*$, the dynamics of the pair $(X,S)$ can be written as
\begin{equation*}
\left\{
\begin{aligned}
\ud X_t&= \mu_0(t, X_t) \ud t + \sigma_0(t,X_t) \ud W^0_t + \int_Z K_0(\zeta; t, X_{t^-}) \N(\ud t,\ud \zeta), \quad X_0=x \in \R\\
\ud S_t&= S_{t^-}\left\{\sigma_1(t, S_t)\ud W^*_t + \int_Z K_1(\zeta;t,X_{t^-}, S_{t^-})\widetilde \N^*(\ud t,\ud \zeta)\right\}, \quad S_0=s
>0,
\end{aligned}
\right.
\end{equation*}
where $W^0,W^*$ are $(\bF, \P^*)$-Brownian motions, with
 \[
 W^*_t=W^1_t+\int_0^tS_u \alpha_u^\F \sigma(u, S_u) \ud u, \ t \in [0,T],
 \]
whose correlation coefficient is $\rho$, $\widetilde \N^*(\ud t,\ud \zeta)$  is the compensated Poisson measure under $\P^*$ given by
\[
\widetilde \N^*(\ud t,\ud \zeta)= \N(\ud t,\ud \zeta)-\eta^*_t(\ud \zeta) \ud t
\]
and $\eta_t^*(\ud \zeta)= (1-\alpha^\F_t S_t K_1(t, X_t, S_t))\eta(\ud \zeta)$ for every $t \in [0,T]$, with $\alpha^\F$ in \eqref{eq:alphaF}.

We will assume that
\begin{gather}
\mathbb{E}^{\P^*}\left[ \int_0^T  \left(|\mu_0(t, X_t)|+\sigma_0^2(t, X_t)+ \eta_t^*(D^0_t)+ \int_Z |K_0(\zeta;t, X_t)|\eta_t^*(\ud \zeta)\right)\ud t\right] <
\infty,\label{integrab3}\\
\mathbb{E}^{\P^*}\left[ \int_0^T   \eta_t^*(D_t)\ud t\right] <
\infty,\label{integrab4}
\end{gather}
where $D^0_t\!=\!\{\zeta\in Z: K_0(\zeta; t, X_{t^-} )\neq 0\}$ and $D_t\!=\!\{\zeta\in Z: K_1(\zeta;t,X_{t^-},S_{t^-})\neq0\}$.

Again, since the change of probability measure is Markovian, the pair $(X,S)$ is still an $(\bF,\P^*)$-Markov process and we provide the structure of its $\P^*$-generator in the following result.
\begin{proposition}\label{lemma-generatore-L3}
Assume \eqref{integrab3}, \eqref{integrab4}.
Then, the pair $(X,S)$ is an $(\bF, \P^*)$-Markov process with generator
\begin{equation} \label{generatore-L3}
\begin{split}
\L^3_{X,S} f(t,x,s) =& \frac{\partial f}{\partial t}+\mu_0(t,x)\frac{\partial f}{\partial x}+ \frac{1}{2} \sigma_0^2(t,x) \frac{\partial^2 f}{\partial
x^2}+ \rho \sigma_0(t,x) \sigma_1(t,s) s  \frac{\partial^2 f}{\partial x\partial s} \\
  &+\frac{1}{ 2} \sigma_1^2(t,s)\, s^2  \frac{\partial^2 f}{\partial s^2} +  \int_Z \Delta f(\zeta;t,x,s)\eta_t^*(\ud \zeta)-\frac{\partial f}{\partial
  s} s \int_Z K_1(\zeta;t,x,s) \eta_t^*(\ud \zeta),
    \end{split}
\end{equation}
where
\[
\Delta f (\zeta;t,x,s)= f \big (t,x+K_0(\zeta;t,x), s( 1 +K_1(\zeta; t,x, s)) \big)-f(t,x,s).
\]
Moreover, the following semimartingale decomposition holds:
\[
f(t, X_t, S_t)= f(0, x_0, s_0) + \int_0^t \L^3_{X,S}f (r,X_r,S_r)\ud r + M^{3, f}_t,
\]
where $M^{3, f}=\{M_t^{3,f},\ t \in [0,T]\}$ is the $(\bF, \P^*)$-martingale given by
 \begin{gather}\label{eq:M3f}
\ud M^{3, f}_t =  \frac{\partial f}{\partial x} \sigma_0(t,X_t) \ \ud W^{0}_t + \frac{\partial f}{\partial s} \sigma_1(t,S_t) S_t \ \ud W^*_t
+ \int_Z \Delta f(\zeta;t,X_{t^-},S_{t^-}) \widetilde \N^*(\ud t, \ud \zeta).
\end{gather}

\end{proposition}
The proof is postponed to Appendix \ref{appendix:b}.

Finally we assume \eqref{hp:filtrazione_esempi}, as in the previous examples. Therefore, taking  the dynamics of the filter for the jump-diffusion case given in \eqref{eq:ks_jumpdiff} in Proposition \ref{prop:ks_jumpdiff}, Appendix \ref{appendix:a}, into account, under \eqref{ass:boundedness_2}, \eqref{integrab3}, \eqref{integrab4}, the $\bH$-pseudo optimal strategy can be written as
\[
\beta^\H_t= \widetilde \beta^\H_t+\phi^\H_t, \quad t \in [0,T]
\]
where
\begin{align*}
\widetilde \beta^\H_t&=\frac{S_{t^-} \sigma_1(t, S_{t^-}) h_{t^-}(g)+\int_\R z  \ w^g(t, z) \nu_t^{\bH,*} (\ud z) }{S^2_{t^-}\sigma_1^2(t, S_{t^-}) +\int_\R z^2\nu_t^{\bH,*} (\ud z) },\quad t \in [0,T],\\
\phi^\H_t&= \frac{\alpha_t^\H \int_\R z^2  \left( \widetilde \beta^\H_t \ z - w^g(t,z) \right) \nu^\bH_t(\ud z) }{S^2_{t^-}\sigma_1^2(t, S_{t^-}) +\int_\R z^2\nu_t^{\bH} (\ud z) },\quad t \in [0,T].
\end{align*}

Here $h_{t}(g)$ and $w^g(t, z)$ are defined in \eqref{eq:h_jumpdiff} and \eqref{eq:w_jumpdiff} in Appendix, respectively, with the choice $f=g$ and $g(t,x,s)$ is
the solution of \eqref{eq:problema} with $\L^*_{X,S}=\L^3_{X,S}$.

\begin{center}
{\bf Acknowledgements}
\end{center}
The authors are grateful to unknown referees for their useful comments and suggestions which lead to a significant improvement of the paper.
The authors also wish to acknowledge the Gruppo Nazionale per l'Analisi Matematica, la Probabilit\`{a} e le loro
Applicazioni (GNAMPA) of the Istituto Nazionale di Alta Matematica (INdAM) for the financial support.
\bibliographystyle{plain}
\bibliography{RM_biblio}

\appendix
\section{The filtering equation}\label{appendix:a}
We recall that the filter with respect to the MMM $\P^*$ is given by
\begin{equation*}
\pi_t(f) =  \mathbb{E}^{\P^*} [ f(t,X_t,S_t) | \F^S_t ], \quad t \in [0,T]
\end{equation*}
for any measurable function $f(t,x,s)$ such that $\espp{|f(t,X_t,S_t)|}< \infty$, for every $t\in [0,T]$.

Here, we will derive the filter dynamics for the jump-diffusion model and deduce the equations for the continuous model and for the
pure jump one, as particular cases. Hence, using the same notations of Section \ref{sec:jumpdiff_model}, we assume that the dynamics  of the pair signal-observation under $\P^*$ is given by
\begin{equation*}
\left\{
\begin{aligned}
\ud X_t & = \mu_0(t, X_t) \ud t + \sigma_0(t,X_t) \ud W^0_t + \int_Z K_0(\zeta; t, X_{t^-}) \N( \ud t, \ud \zeta),\quad X_0=x \in \R \\
\ud S_t & = S_{t^-}\left\{\sigma_1(t, S_t)\ud W^*_t + \int_Z K_1(\zeta;t,X_{t^-},S_{t^-}) \widetilde{\N}^*(\ud t,\ud \zeta)\right\}, \quad
S_0=s>0.
\end{aligned}
\right.
\end{equation*}
We assume \eqref{integrab3} and \eqref{integrab4}, in addition to \eqref{ass:boundedness_2}, which in particular imply that the processes $X$ and $S$ have finite first moment under $\P^*$.
We recall that the jump part of the process $S$ can be described by the integer-valued random measure $m(\ud t,\ud z)$ defined in \eqref{m}. We denote
by
$\nu_t^{\bF,*}(\ud z) \ud t$ its $(\bF,\P^*)$-predictable dual projection and by $\nu_t^{\bH,*}(\ud z)\ud t$ its $(\bH, \P^*)$-predictable dual
projection and the following relationship holds
\[
\nu_t^{\bH,*}(\ud z)\ud t=\pi_{t^-}(\nu^{\bF,*}(\ud z))\ud t
\]
thanks to~\cite[Proposition 2.2]{c06}.

\begin{remark}\label{rem:martingale_representation_theorem}
An essential tool to derive the filtering equation is represented by the Martingale Representation Theorem (see~\cite[Proposition 2.6]{cco1}). In
particular, it states that every $(\bH, \P^*)$-local martingale admits the following representation
\[
M_t=M_0+\int_0^th_u \ud I^*_t+ \int_0^t \int_{\R} w(u,z) \left(m(\ud u, \ud z )-\nu_u^{\bH,*}(\ud z)\ud t\right),\quad t \in [0,T],
\]
for suitable $\bH$-adapted  and $\bH$-predictable processes $h=\{h_t, \ t \in [0,T]\}$ and $w(\cdot, z)=\{w(t,z), \ t \in [0,T]\}$ for every $z \in
\R$, satisfying
\[
\int_0^T\left(h^2_t+\int_\R |w(t,z)| \nu_t^{\bH,*}(\ud z)\right) \ud t < \infty \quad \P^*-\mbox{a.s.}
\]
where $I^*=\{I^*_t, \ t \in [0,T]\}$ is the $(\bH, \P^*)$-Brownian motion given by
\[
I^*_t=W^*_t+ \int_0^t\left\{\frac{b(u, X_u, S_u)}{\sigma_1(u, S_u)} - \pi_u\left(\frac{b}{\sigma_1}\right)\right\}\ud u,\quad t \in [0,T]
\]
with $\ds b(t, X_t, S_t)= \int_ZK_1(\zeta; t , X_t, S_t) \ \eta_t^*(\ud \zeta)$ for every $t \in [0,T]$.
\end{remark}
The following result provides the dynamics of the filter.
\begin{proposition}[The filtering equation] \label{prop:ks_jumpdiff}
Under \eqref{ass:boundedness_2}, \eqref{integrab3} and \eqref{integrab4} 
the filter solves the Kushner-Stratonovich
equation for every function $f(t,x,s) \in \C_b^{1,2,2}([0,T] \times \mathbb{R} \times \R^+)$, given by
\begin{equation} \label{eq:ks_jumpdiff}
\pi_t (f) = f(0,x_0,s_0) + \int_0^t \pi_s(\L^3_{X,S} f) \ud s + \int_0^t h_s(f)  \ud I^*_s + \int_0^t \int_\R w^f(s,z) (m(\ud s, \ud
z)-\nu_s^{\bH,*}(\ud z)\ud s),
\end{equation}
for every $t \in [0,T]$, where
\begin{equation} \label{eq:h_jumpdiff}
 h_t(f)=  \rho \pi_{t}\left(\sigma_0 \frac{\partial f }{\partial x}\right) + S_t \sigma_1(t,S_t) \pi_t \left(\frac{\partial f }{\partial
 s}\right),
 \end{equation}
\begin{equation}\label{eq:w_jumpdiff}
w^f(t,z)= \frac{\ud \pi_{t^-} (f \nu^{\bF,*})}{\ud \nu^{\bH,*}_t} (z) - \pi_{t^-}(f) +
  \frac{ \ud \pi_{t^-} (\overline{\L} f)}{\ud \nu^{\bH,*}_{t}} (z),
\end{equation}
$\L^3_{X,S}$ is given in  \eqref{generatore-L3} and $\overline{\L}f(t,x,s,\A):=\int_{d^\A(t,x,s)}\{f(t,x+K_0(\zeta;t,x),
s(1+K_1(\zeta;t,x,s)))-f(t,x,s)\}\eta_t^*(\ud \zeta)$,  for every $ \A   \in\mathcal{B}(\mathbb{R})$, where $d^\A(t,x,s)=\{\zeta \in Z:
K_1(\zeta;t,x,s)\in \A\setminus \{0\}\}$.
\end{proposition}
\begin{proof}
We consider the semimartingale $Z=\{Z_t=f(t,X_t,S_t),\ t \in [0,T]\}$ whose decomposition is given by
\begin{equation}\label{zt}
 Z_t=f(0,X_0,S_0)+\int_0^t \L^3_{X,S}f(u,X_u,S_u)\ud u+M^{3,f}_t,\quad t \in [0,T],
\end{equation}
where $\L^3_{X,S}$ and $M^{3,f}$ are defined in \eqref{generatore-L3} and \eqref{eq:M3f}, respectively.
By taking the conditional expectation with respect to $\H_t$ in \eqref{zt},  we get
\begin{equation*}
\pi_t(f)=\pi_0(f)+\int_0^t\pi_u(\L^3_{X,S}f)\ud u+\widetilde{M}^{f}_t,\quad t \in [0,T],
\end{equation*}
where $\ds \widetilde{M}^f_t={}^{o,*} M_t^{3,f}+{}^{o,*}\left(\int_0^t \L^3_{X,S} f (u,X_u,S_u) \ud u\right)-\int {}^{o,*}(\L^3_{X,S} f(u,X_u,S_u))\ud
u $ for every $t \in [0,T]$, and ${}^{o,*} Y$ is the $(\bH, \P^*)$-optional projection of a given process $Y$,  and $\widetilde{M}^f$ is an
$(\bH,\P^*)$-martingale. Thanks to Remark \ref{rem:martingale_representation_theorem}, there exist an $\bH$-adapted process $h(f)$ and an
$\bH$-predictable process $w^f$ such that
$$
 \espp{\int_0^T \left(h_s^2(f)+\int_{\R}|w^f(s,z)|\nu_s^{\bH,*}(\ud z)\right) \ud s} < \infty
$$
 and
\begin{equation*}
\widetilde{M}^f_t=\pi_t(f)-\pi_0(f)-\int_0^t \pi_u(\L^3_{X,S} f)\ud u= \int_0^t h_u(f)\ud I^*_u+\int_0^t\int_{\R}w^f(u,z)\left(m(\ud u, \ud
z)-\nu_u^{\bH,*}(\ud z)\ud u\right).
\end{equation*}
To identify the process $h(f)$ we define the process $\widetilde W^*=\{\widetilde W^*_t, \ t \in [0,T]\}$ by
\[
\widetilde W^*_t:=I^*_t+ \int_0^t\frac{\pi_u(b)}{\sigma_1(u, S_u)} \ud u,\quad t \in [0,T].
\]
Then we compute ${}^{o,*}\left( Z \widetilde{W}^*\right)$ and $ {}^{o,*}Z\widetilde{W}^*$ separately and since $\widetilde{W}^*$ is
$\bH$-adapted, the equality $ {}^{o,*}\left( Z \widetilde{W}^*\right)={}^{o,*}Z\widetilde{W}^*$ holds.
By It\^o's product rule, we have
\begin{equation*}
\ud (Z_t  \widetilde{W}^*_t) = Z_t \ud \widetilde{W}^*_t+ \widetilde{W}^*_t \ \L^3_{X,S}f(t,X_t,S_t)\ud t +\frac{\partial f}{\partial x} \sigma_0(t,X_t) \rho \ud t +
\frac{\partial f}{\partial s} S_t \sigma_1(t,S_t) \ \ud t+\ud M^1_t,
\end{equation*}
where $\displaystyle M^1 := \int  \widetilde{W}^*_s \ud M^{3, f}_s$ is an $( \bF, \P^*)$-local martingale. We now introduce an $ \bH$-localizing sequence for
$M^1$:
\[
\widetilde{\tau}_n = T\wedge \inf\left\{t:  | \widetilde{W}^*_t| \geq n\right\}, \quad n\geq 1.
\]
If we take the conditional expectation with respect to $\H_t$, on $\{t \leq \widetilde{\tau}_n\}$ we get
\begin{eqnarray*}
\ud{}^{o,*} (Z_t \widetilde{W}^*_t)&=& {}^{o,*}\left(\widetilde{W}^*_t \ \L^3_{X,S}f(t,X_t,S_t) +\frac{\partial f}{\partial x}(t) \sigma_0(t,X_t) \rho +
\frac{\partial f}{\partial s} S_t \sigma_1(t,S_t) \right) \ud t +  \ud \widetilde M^1_t,
 \end{eqnarray*}
where $\widetilde M^1$ is  an $(\bH, \P^*)$-local martingale.
On the other hand
\begin{equation*}
\ud ({}^{o,*} Z_t \widetilde{W}^*_t)=\left(\widetilde{W}^*_t \ {}^{o,*}\L^3_{X,S}f(t,X_t,S_t)+h_t(f)\right)\ud t+\ud M^2_t,
\end{equation*}
where $ M^2$ is an $(\bH, \P^*)$-local martingale.
By the equality ${}^{o,*}\left( Z\widetilde{W}^*\right)={}^{o,*}Z\widetilde{W}^*$, the bounded variation terms must be equal, which means that
\begin{equation*}
 h_t(f)=  \rho  \pi_{t}\left(\sigma_0  \frac{\partial f }{\partial x}\right) + S_t \sigma_1(t,S_t)\pi_t \left(\frac{\partial f }{\partial
 s}\right)
\end{equation*}
on $\{t \leq  \widetilde{\tau}_n\}$.
Now,  when $n \to \infty$, $\widetilde{\tau}_n$ goes to $T$ $\P$-a.s. and so the process $h_t(f)$ is completely defined for every $t\in [0,T]$.
Following the same arguments of the proof of Theorem 3.2 in~\cite{cco1} we obtain the expression of $w^f(t,z)$.
\end{proof}

\begin{remark}
Strong uniqueness for the solution of the filtering equation is analyzed in~\cite{cco1} and~\cite{cco2} for the pair signal-observation given by the
system \eqref{eq:sistema3}. These results can be applied to deduce suitable conditions which ensure strong uniqueness of the solution to the filtering
equation \eqref{eq:ks_jumpdiff} under the MMM $\P^*$. In~\cite{cco2} the authors analyzed strong uniqueness for the Zakai equation solved by the
unnormalized version of the filter, and the relation with pathwise uniqueness for the Kushner-Stratonovich equation. In particular, whenever the signal
process $X$ is a pure jump process taking values in a countable space, the Zakai equation can be solved recursively (see Section 5.3 in~\cite{cco2}
and~\cite{cg2000}) and pathwise uniqueness holds under the hypothesis that $X$ takes values in a finite space or when $X$ and $S$ have only common jump
times.
\end{remark}

\begin{remark}[The diffusion market model]\label{rem:ks_continuo}
If the dynamics of the pair $(X,S)$ is given by the system \eqref{eq:sistema1b}, we observe that
\[
W^*_t=W^1_t+\int_0^t\frac{\mu_1(u, X_u, S_u)}{\sigma_1(u , S_u)} \ud u = \widetilde W_t, \quad t \in [0,T],
\]
then $I^*_t=I_t$, with $I_t$ given in \eqref{eq:innovation}, for every $t \in [0,T]$.
Therefore under \eqref{ass:boundedness_1} and \eqref{hp:generatore_cont},
the dynamics of the filter becomes
\begin{equation}\label{eq:ks_continuo}
\pi_t (f) = f(0,x_0,s_0) + \int_0^t \pi_s(\L^1_{X,S} f) \ud s + \int_0^t h_s(f)  \ud I_s   \end{equation}
for every $t \in [0,T]$ and for every function $f \in \C^{1,2,2}([0,T]\times \R \times \R^+)$, where
\begin{equation}\label{eq:h_continuo}
 h_t(f)=  \rho \pi_{t}\left(\sigma_0 \frac{\partial f }{\partial x}\right) + S_t \sigma_1(t,S_t) \pi_t \left(\frac{\partial f }{\partial
 s}\right), \quad t \in [0,T]. \end{equation}
\end{remark}

\begin{remark}[The pure jump market model]\label{rem:ks_jump}
For the pure jump market model described by the system \eqref{eq:sistema2b}, under \eqref{ass:boundedness_1b}, \eqref{integrab} and \eqref{hp:L2} the filter dynamics is given by
\begin{equation} \label{eq:ks_jump}
\pi_t (f) = f(0,x_0,s_0) + \int_0^t \pi_s(\L^2_{X,S} f) \ud s + \int_0^t \int_\R w^f(s,z) (m(\ud s, \ud z)-\nu^{\bH,*}_s(\ud z)\ud s), \end{equation}
where
\begin{equation*}\label{eq:w_jump}
w^f(t,z)= \frac{\ud \pi_{t^-} (f \nu^{\bF,*})}{\ud \nu^{\bH,*}_t} (z) - \pi_{t^-}(f) +
  \frac{ \ud \pi_{t^-} (\overline{\L} f)}{\ud \nu^{\bH,*}_{t}} (z).
\end{equation*}
In ~\cite{cg06} an explicit representation of the filter is obtained by the Feynman-Kac formula using a linearization method. This representation allows one to provide a recursive algorithm for the computation of the filter.
\end{remark}

\subsection{The filtering equation under the real-world probability measure}

As pointed out in Section \ref{sec:jump_model}, to derive the $\bH$-pseudo optimal strategy we also need to compute $\nu_t^\bH(\ud z)\ud t$ which is
the $(\bH, \P)$-predictable dual projection of the integer valued random measure $m(\ud t, \ud z)$. We observed that $\nu_t^\bH(\ud z)\ud t$ has a
representation in terms of $\widetilde \pi$ which is the filter under the real-world probability measure $\P$, given by $\nu_t^\bH(\ud z)= \widetilde
\pi_{t^-}(\nu^\bF(\ud z))$.

Under \eqref{ass:boundedness_2} and \eqref{integrab3}, \eqref{integrab4},
formulated under $\P$, by extending the results in~\cite{cco1}, the filter $ \widetilde \pi$ solves the following Kushner-Stratonovich equation
\begin{equation} \label{eq:ks_P}
\widetilde \pi_t (f) = f(0,x_0,S_0) + \int_0^t \widetilde \pi_s(\L^{X,S} f) \ud s +
\int_0^t\int_{\mathbb{R}}  \widetilde w^f(s,z) ( m(\ud s, \ud z)  -  \widetilde \pi_{s^-}(\nu^\bF(\ud z)))+ \int_0^t  \widetilde  h_s(f)  \ud I_s
\end{equation}
for every function $f \in \mathcal{C}^{1,2}([0,T] \times \R \times \R^+)$ and for every $t\in [0,T]$, where
\begin{equation*} \label{d1}
  \widetilde w^f(t,z)= \frac{\ud \widetilde \pi_{t^-} (\nu^\bF f) }{  \ud\widetilde \pi_{t^-} (\nu^\bF )} (z) - \widetilde \pi_{t^-}(f) +
  \frac{ \ud \widetilde\pi_{t^-} (\widetilde{L} f)}{  \ud \widetilde \pi_{t^-} \left(\nu^\bF\right)} (z), \quad t \in [0,T] \end{equation*}

 \begin{equation*} \label{d2}
  \widetilde h_t(f)=  \frac{\widetilde \pi_{t}( \mu_1 f) -\widetilde  \pi_{t}(\mu_1) \widetilde \pi_{t}(f)}{\sigma_1(t, S_t)} + \rho \widetilde
  \pi_{t}\left(\sigma_0  \frac{\partial f}{\partial x}\right) + S_t \sigma_1(t,S_t) \widetilde \pi_t \left(\frac{\partial f }{\partial s}\right), \quad t \in [0,T],
  \end{equation*}
  and $I$ is the innovation process defined by \eqref{eq:innovation}. The operator  $\L^{X,S}$ denotes the generator of $(X,S)$ under $\P$, which is given by
 \[
 \begin{aligned}[t]
 &\L^{X,S}f(t,x,s)= \frac{\partial f}{\partial t}+\frac{\partial f}{\partial x} \mu_0(t,x)+\frac{\partial f}{\partial
 s} s \mu_1(t,x,s) +\frac{1}{2} \frac{\partial^2 f}{\partial x^2} \sigma_0^2(t,x) + \frac{1}{2}\frac{\partial^2 f}{\partial s^2} s^2
 \sigma_1^2(t,s)\\
 &\quad +\frac{\partial^2 f}{\partial x \partial s}  s \sigma_0(t,x) \sigma_1(t,s) \rho + \int_Z \Delta f(\zeta;t,x,s) \eta(\ud \zeta)
 \end{aligned}
 \]
with $\ds \Delta f (\zeta; t,x,s)= f \big (t,x+K_0(\zeta;t,x), s( 1 +K_1(\zeta; t,x, s)) \big)-f(t,x,s)$,  and  for every $ \A
\in\mathcal{B}(\mathbb{R})$, the operator $\widetilde{\L}$, defined by
\begin{equation*}\label{operatoreL}
\widetilde{\L} f(t,x,s,\A) :=  \int_{d^\A(t,x,s)}\Delta f (\zeta; t,x,s)\eta(\ud \zeta),
\end{equation*}
where $d^\A(t,x,s)=\{\zeta \in Z: K_1(\zeta;t,x,s)\in \A\setminus\{0\}\}$,  takes common jump times between the signal $X$ and the observation $S$ into
account.

\begin{remark}[The pure jump market model]
Clearly, we can deduce the filtering equation for the pure jump model as a particular case of equation \eqref{eq:ks_P}, where now we have $\widetilde
h(f)=0$ and
$$ \L^{X,S}f(t,x,s)= \frac{\partial f}{\partial t}+\frac{\partial f}{\partial x} \mu_0(t,x)
+\frac{1}{2} \frac{\partial^2 f}{\partial x^2} \sigma_0^2(t,x) + \int_Z \Delta f(\zeta;t,x,s)  \eta(\ud \zeta).
$$
\end{remark}

\section{Some proofs}\label{appendix:b}

\begin{proof}[Proof of Lemma \ref{lemma:MMM}]

We denote by $(B^\bF, C^\bF, \nu^\bF)$ the $(\bF, \P)$-predictable characteristics of $S$ (see \cite{js} for more details) and by $(B^\bH, C^\bH, \nu^\bH)$ the $(\bH, \P)$-predictable characteristics of $S$.

Assume now that $S$ has continuous trajectories, then $\nu^\bF=\nu^\bH=0$. Then the $(\bF, \P)$-predictable characteristics of $S$ are given by

\[
B^\bF_t=\int_0^t\alpha^\F_u \ud \langle M \rangle_u \qquad C^\bF_t:=\langle S \rangle_t=\langle M \rangle_t, \quad t \in [0,T],
\]
and the $(\bH, \P)$-predictable characteristics of $S$ are

\[
B^\bH_t=\int_0^t\alpha^\H_u \ud \langle N \rangle_u \qquad C^\bH_t:=\langle S \rangle_t=\langle N \rangle_t, \quad t \in [0,T].
\]

We also recall that in the continuous trajectories case we also get that $\alpha^\H={}^p(\alpha^\F)$, and $\langle S\rangle=\langle M\rangle=\langle N\rangle$.

This means that the $(\bH,\P)$-predictable characteristics of $S$ can also be written as
\[
B^\bH_t=\int_0^t{}^p\alpha^\F_u \ud \langle M \rangle_u \qquad C^\bH_t:=\langle S \rangle_t=\langle M \rangle_t.
\]

Using the definition of $S$ we get that
\[
S_t-S_0=M_t+\int_0^t\alpha^\F_u \ud \langle M \rangle_u=N_t+\int_0^t\alpha^\H_u \ud \langle N \rangle_u,\quad t \in [0,T].
\]

Hence, by the Girsanov theorem we get that $S$ has $(\bF, \P^*)$-predictable characteristics $(0, \langle M \rangle, 0)$ and since $\langle M \rangle =\langle N \rangle$, this are also the $(\bH, \P^*)$-predictable characteristics of $S$.

Again, by the Girsanov theorem $S$ has $(\bH, \P^0)$-predictable characteristics given by $(0, \langle N \rangle, 0 )$.

Therefore since $(\bH, \P^*)$-predictable characteristics of $S$ coincide with its $(\bH, \P^0)$-predictable characteristics and $\P^0|_{\H_0}=\P^*|_{\H_0}$, by \cite[Chapter 3, Corollary 4.31]{js} we can conclude that $\P^0$ is the restriction of $\P^*$ over $\bH$.
\end{proof}

\begin{proof}[Proof of Proposition \ref{lemma-generatore-Lcont}]
Observe that the change of probability measure $\ds \left.\frac{\ud \P^*}{\ud \P}\right|_{\F_T}$ is Markovian since $\alpha^\F_t=\alpha^\F(t, X_{t^-},
S_{t^-})$, for each $t \in [0,T]$ (see~\cite[Proposition 3.4]{cg09}). Then the pair $(X,S)$ is still an $(\bF,\P^*)$-Markov process. To compute the
generator $\L^*_{X,S}$, we apply It\^{o}'s formula to the function $f(t,X_t,S_t)$, and we get
$$
f(t, X_t, S_t)= f(0, x_0, s_0) + \int_0^t \L^1_{X,S}f (r,X_r,S_r) \ud r + M^{1,f}_t,
$$
 where $\ds \L^1_{X,S}$ is the operator given in \eqref{generatore-L1} and $M^{1,f}$ is the process given by \eqref{eq:M1f}.
Moreover, under \eqref{hp:generatore_cont} the process $M^{1,f}$ is an $(\bF, \P^*)$-martingale; indeed
\[
\bE^{\P^*}\left[\int_0^T \!\!\! \sigma^2_0(t,X_t) \left( \frac{\partial f}{\partial x} \right)^2\ud t\right] < \infty,  \qquad \bE^{\P^*}\left[
\int_0^T \!\!\!  \sigma_1^2(t,S_t) S^2_t \left( \frac{\partial f}{\partial s} \right)^2 \ud t\right] < \infty.
\]
\end{proof}

\begin{proof}[Proof of Proposition \ref{lemma-generatore-L2}]
By the same argument used in the proof of Proposition \ref{lemma-generatore-Lcont}, we get that
\[
f(t, X_t, S_t)= f(0, x_0, s_0) + \int_0^t \L^2_{X,S}f(r,X_r,S_r) \ud r + M^{2,f}_t,
\]
where $\ds \L^2_{X,S}$ is the operator in \eqref{generatore-L2} and $M^{2,f}$ is given by \eqref{eq:Mf2}.
Note that under conditions \eqref{integrab} and \eqref{hp:L2}, the process $M^{2,f}$ is an $(\bF,\P^*)$-martingale; indeed
\[
\bE^{\P^*}\left[\int_0^T \!\!\! \sigma^2_0(t,X_t) \left( \frac{\partial f}{\partial x} \right)^2\ud t\right] < \infty,
\]
and
$$
\bE^{\P^*}\left[\int_0^T \int_Z |\Delta f(\zeta;t,X_{t^-},S_{t^-})| \eta_t^*(\ud \zeta)\ud t\right]
\leq 2 \|f\|\bE^{\P^*}\left[\int_0^T \{\eta_t^*(D^0_t) + \eta_t^*(D_t)\} \ud t\right]  < \infty,
$$
where $\|f\|=\sup\{f(t,x,s)|(t,x,s) \in  \R^+ \times \R \times \R^+\}$.
\end{proof}

\begin{proof}[Proof of Proposition \ref{lemma-generatore-L3}]
Analogously to the proof of Proposition \ref{lemma-generatore-Lcont}, we get that
\[
f(t, X_t, S_t)= f(0, x_0, s_0) + \int_0^t \L^3_{X,S}f (r,X_r,S_r)\ud r + M^{3,f}_t,
\]
where $\ds \L^3_{X,S}$ is the operator in \eqref{generatore-L3} and $M^{3,f}$ is given in \eqref{eq:M3f}.
Moreover, under conditions \eqref{integrab3}, \eqref{integrab4}, 
the process $M^{3,f}$ is an $(\bF, \P^*)$-martingale; indeed
\[
\bE^{\P^*}\left[\int_0^T \!\!\! \sigma^2_0(t,X_t) \left( \frac{\partial f}{\partial x} \right)^2\ud t\right] < \infty,  \qquad \bE^{\P^*}\left[
\int_0^T \!\!\!  \sigma_1^2(t,S_t) S^2_t \left( \frac{\partial f}{\partial s} \right)^2 \ud t\right] < \infty
\]
and
$$
\bE^{\P^*}\left[\int_0^T \int_Z |\Delta f(\zeta;t,X_{t^-},S_{t^-})| \eta_t^*(\ud \zeta)\;\ud t\right]
\leq 2 \|f\| \bE^{\P^*}\left[\int_0^T \{ \eta_t^*(D^0_t) + \eta_t^*(D_t)\} \ud t\right]  < \infty .
$$
\end{proof}

\end{document}